\documentclass[11pt]{article}
\usepackage[utf8]{inputenc}
\usepackage{amsfonts,latexsym,amsthm,amssymb,amsmath,amscd,euscript,mathrsfs,graphicx}
\usepackage{framed}
\usepackage{fullpage}
\usepackage{hyperref}
    \hypersetup{colorlinks=true,citecolor=blue,urlcolor =black,linkbordercolor={1 0 0}}
\usepackage{enumitem}
\usepackage{color}
\newcommand{\Mod}[1]{\ (\mathrm{mod}\ #1)}

\allowdisplaybreaks[1]

\newtheorem{theorem}{Theorem}
\newtheorem{proposition}{Proposition}[section]
\newtheorem{lemma}{Lemma}[section]
\newtheorem{corollary}{Corollary}[section]

\theoremstyle{definition}
\newtheorem{defn}{Definition}[section]

\newcommand{\BE}{\mathbb{E}}
\newcommand{\BF}{\mathbb{F}}
\newcommand{\BP}{\mathbb{P}}
\newcommand{\BR}{\mathbb{R}}
\newcommand{\BN}{\mathbb{N}}
\newcommand{\BZ}{\mathbb{Z}}

\theoremstyle{remark}
\newtheorem*{remark}{Remark}

\title{3-wise Independent Random Walks can be Slightly Unbounded}
\author{Shyam Narayanan}
\author{Shyam Narayanan\thanks{\texttt{shyam.s.narayanan@gmail.com}. Massachusetts Institute of Technology. Research supported by Harvard's Herchel Smith Fellowship and the MIT Akamai Fellowship.}}

\begin{document}

\maketitle

\begin{abstract}

    Recently, many streaming algorithms have utilized generalizations of the fact that the expected maximum distance of any $4$-wise independent random walk on a line over $n$ steps is $O(\sqrt{n})$. In this paper, we show that $4$-wise independence is required for all of these algorithms, by constructing a $3$-wise independent random walk with expected maximum distance $\Omega(\sqrt{n} \lg n)$ from the origin. We prove that this bound is tight for the first and second moment, and also extract a surprising matrix inequality from these results.
    
    Next, we consider a generalization where the steps $X_i$ are $k$-wise independent random variables with bounded $p$th moments. For general $k, p$, we determine the (asymptotically) maximum possible $p$th moment of the supremum of $X_1 + \dots + X_i$ over $1 \le i \le n$. We highlight the case $k = 4, p = 2$: here, we prove that the second moment of the furthest distance traveled is $O(\sum X_i^2)$. For this case, we only need the $X_i$'s to have bounded second moments and do not even need the $X_i$'s to be identically distributed. This implies an asymptotically stronger statement than Kolmogorov's maximal inequality that requires only $4$-wise independent random variables, and generalizes a recent result of B{\l}asiok.
\end{abstract}


\newpage

\section{Introduction}

Random walks are well-studied stochastic processes with numerous applications. The simplest such random walk is the random walk on $\BZ$, a process that starts at $0$ and at each step independently of all previous moves moves either $+1$ or $-1$ with equal probability. In this paper, we do not study this random walk but instead study $k$-wise independent random walks, meaning that steps are not totally independent but that any $k$ steps are completely independent. In many low-space randomized algorithms, information is tracked with processes similar to random walks, but simulating a totally random walk of $n$ steps is known to require $O(n)$ bits while there exist $k$-wise independent families which can be simulated with $O(k \lg n)$ bits \cite{kWise}. As a result, understanding properties of $k$-wise independent random walks have applications to streaming algorithms, such as heavy-hitters \cite{BPTree, CountSieve}, distinct elements \cite{Jarek}, and $\ell_p$ tracking \cite{LpContinuousMonitor}.

For any $k$-wise independent random walk, where $k \ge 2$, it is well-known that after $n$ steps, the expected squared distance from the origin is exactly $n$, since $\BE_{h \sim \mathcal{H}} (h(1)+\dots+h(n))^2 = n$ for any $2$-wise independent hash family $\mathcal{H}$. One can see this by expanding and applying linearity of expectation. This property provides good bounds for the distribution of the final position of a $2$-wise independent random walk. However, we study the problem of bounding the position throughout the random walk, by providing comparable moment bounds for $\sup_{1 \le i \le n} |h(1)+\dots+h(i)|$ rather than just for $|h(1)+\dots+h(n)|$. Importantly, we determine an example of a $2$-wise independent random walk where the expected bounds do not hold, even though very strong bounds for even $4$-wise independent random walks can be established.

Two more general questions that have been studied in the context of certain streaming algorithms are random walks corresponding to insertion-only streams, and random walks with step sizes corresponding to random variables. These are useful generalizations as the first proves useful in certain algorithms with insertion stream inputs, and the second allows for a setup similar to Kolmogorov's inequality \cite{Kolmogorov}, which we will generalize to $4$-wise independent random variables. To understand these two generalizations, consider a $k$-wise independent family of random variables $X_1, \dots, X_n$ and an insertion stream $p_1, \dots, p_m \in [n]$, where now seeing $p_j$ means that our random walk moves by $X_{p_j}$ on the $j$th step. The insertion stream can be thought of as keeping track of a vector $z$ in $\BR^n$ where seeing $p_j$ increments the $p_j$th component of $z$ by $1$, and $\vec{X}$ can be thought of as a vector in $\BR^n$ with $i$th component $X_i.$  Then, one goal is to bound, for appropriate values of $p$,
\[\BE_{h \sim \mathcal{H}} \left[\sup\limits_{1 \le t \le m} \left|\langle \vec{X}, z^{(t)}\rangle\right|^{p} \right],\]
where $z^{(t)}$ is the vector $z$ after seeing only the first $t$ elements of the insertion stream. Notice that bounding the $k'$th moment of the furthest distance from the origin in a $k$-wise independent random walk is the special case of $m = n$, $p_j = j$ for all $1 \le j \le n,$ and the $X_i$'s are $k$-wise independent random signs.

\subsection{Definitions} \label{Definitions}

To formally describe our main results, we need to formally define $k$-wise independent random variables, $k$-wise independent hash families, and $k$-wise independent random walks.

\begin{defn}
    For an integer $k \ge 2,$ we say that random variables $X_1, \dots, X_n$ are \emph{k-wise independent} if for any $1 \le i_1 < i_2 < \dots < i_k \le n,$ the variables $X_{i_1}, \dots, X_{i_k}$ are totally independent. Moreover, if $\BE[X_i] = 0$ for all $i$, we say that $X_1, \dots, X_n$ are \emph{unbiased k-wise independent} random variables.
\end{defn}

\begin{defn}
    For an integer $k \ge 2,$ we say that $\mathcal{H}$, which is a distribution over functions $h: \{1, \dots, n\} \to \{-1, 1\}$ is a \emph{k-wise independent hash family} from $\{1, \dots, n\}$ to $\{-1, 1\}$ if the random variables $h(1), \dots, h(n)$ for $h \sim \mathcal{H}$ are unbiased $k$-wise independent random variables.
\end{defn}

For simplicity of notation, we will often treat $h$ as a vector in $\{-1, 1\}^n,$ with $h_i := h(i).$ From now on, we will also use $[n]$ to denote the set $\{1, \dots, n\}.$

\begin{defn}
    Suppose $h: [n] \to \{-1, 1\}$ is drawn from a $k$-wise independent hash family. Then, we say that the random variables $Z_0 = 0$ and $Z_t = h_1 + \dots + h_t$ for $1 \le t \le n$ is a \emph{$k$-wise independent random walk}. Equivalently, a $k$-wise independent random walk is a series of random variables $Z_0, Z_1, Z_2, \dots, Z_n$ where $Z_0 = 0$ and the variables $\{Z_1-Z_0, Z_2-Z_1, \dots, Z_n-Z_{n-1}\}$ are unbiased $k$-wise independent random variables that always equal $\pm 1$.
\end{defn}

\begin{remark}
    We will sometimes refer to $2$-wise independent hash families or random walks as \emph{pairwise independent} hash families or random walks.
\end{remark}

We also formally define an insertion-only stream.

\begin{defn}
    An \emph{insertion-only stream} $p_1, \dots, p_m$ is a sequence of integers where each $p_i \in [n].$ The insertion-only stream will keep track of a vector $z$ that starts at the origin and after the $j$th step, increments its $p_j$th component by $1$. In other words, for any $0 \le t \le m,$ $z^{(t)},$ i.e. the position of $z$ after $t$ steps, will be the vector in $\BR^n$ such that $z^{(t)}_i = |\{j \le t: p_j = i\}|.$
\end{defn}

\subsection{Main Results} \label{MainResults}

Intuitively, even in a pairwise independent random walk, since the positions at various times have strong correlations with each other, the expectation of the furthest distance traveled from the origin should not be much more than the expectation of than the distance from the origin after $n$ steps, which is well-known to be $O(\sqrt{n})$. But surprisingly, we show in Section \ref{LowerBounds} that one can have a 3-wise independent random walk with expected maximum distance $\Omega(\sqrt{n} \log n)$: in fact, our random walk will even have maximum distance $\Omega(\sqrt{n} \log n)$ with constant probability. Formally, we prove the following theorem:
\begin{theorem} \label{Bound1}
There exists an unbiased 3-wise independent hash family $\mathcal{H}$ from $[n]$ to $\{-1, 1\}$ such that
\begin{equation}
    \BE_{h \sim \mathcal{H}} \left[\sup\limits_{1 \le t \le n} \left|h_1+\dots+h_t\right|\right] = \Omega\left(\sqrt{n} \lg n\right).
\end{equation}
In fact, $\mathcal{H}$ can even have $\BP_{h \sim \mathcal{H}}\left(\sup_{1 \le t \le n} |h_1 + \dots + h_t| = \Omega(\sqrt{n} \log n)\right)$ be at least some constant.
\end{theorem}

Furthermore, this bound of $\sqrt{n} \lg n$ is tight up to the first and second moments for pairwise (and thus $3$-wise) random walks, due to the following result we prove in Section \ref{PairwiseUpperBound}.
\begin{theorem} \label{Bound2}
For any unbiased pairwise independent family $\mathcal{H}$ from $[n]$ to $\{-1, 1\}$,
\begin{equation}
    \BE_{h \sim \mathcal{H}} \left[\sup\limits_{1 \le t \le n} \left(h_1+\dots+h_t\right)^2\right] = O\left(n \lg^2 n\right).
\end{equation}
\end{theorem}

In section \ref{UpperBounds}, we bound random walks corresponding to insertion-only streams and random walks with step sizes that are not necessarily uniform $\pm 1$ variables. We first generalize Kolmogorov's inequality \cite{Kolmogorov} to $4$-wise independent random variables via the following theorem.

\begin{theorem} \label{Bound3}
For any unbiased $4$-wise independent random variables $X_1, \dots, X_n$, each with finite variance,
\begin{equation}
    \BE \left(\sup\limits_{1 \le i \le n} |X_1+\dots+X_i|^2 \right) = O\left(\sum \BE[X_i^2]\right).
\end{equation}
\end{theorem}
Applying Chebyshev's inequality gives us Kolmogorov's inequality up to a constant. We remark that if we remove the $4$-wise independence assumption, Theorem \ref{Bound3} can be more simply proved as a consequence of the Skorokhod embedding theorem \cite{Skorokhod}. This weaker version of Theorem \ref{Bound3} may exist in the literature, though as I am not aware of it, I provide a short proof in Appendix \ref{Sheffield}.

Next, we generalize Theorem \ref{Bound2} to more general random variables and insertion-only streams.
\begin{theorem} \label{Bound4}
For any unbiased pairwise independent variables $X_1, \dots, X_n$, each satisfying $\BE[X_i^2] \le 1,$ and for any insertion stream $p_1, \dots, p_m \in [n],$
\begin{equation}
    \BE \left[\sup\limits_{1 \le t \le m} \left|\langle \vec{X}, z^{(t)}\rangle\right|^2\right] = O\left(\|z\|_2^2 \lg^2 m\right)
\end{equation}
where $z = z^{(m)}$ is the final position of the vector and $\vec{X} = (X_1, \dots, X_n)$.
\end{theorem}

Finally, we prove an analogous theorem for $k$-wise independent random variables for even $k \ge 4.$ In this case, the log factors will disappear.
\begin{theorem} \label{Bound5}
For any even $k \ge 4$, any unbiased $k$-wise independent random variables $X_1, \dots, X_n$ such that $\BE[X_i^k] \le 1$ for all $i$, and any insertion stream $p_1, \dots, p_m \in [n],$
\begin{equation}
    \BE \left[\sup\limits_{1 \le t \le m} \left|\langle \vec{X}, z^{(t)}\rangle\right|^k\right] = O(k)^{k/2} \cdot \|z\|_2^k.
\end{equation}
\end{theorem}
We note that as $k$ grows, this is tight, as even $\BE \left[\langle \vec{X}, z^{(m)}\rangle^k\right] = \Theta(k)^{k/2} \cdot \|z\|_2^k$ if $X_1, \dots, X_n$ are totally independent $\pm 1$-valued random variables, as shown in \cite{KhintchineGeneral}.

Theorems \ref{Bound3}, \ref{Bound4}, and \ref{Bound5} are interesting together as they provide various bounds on the supremum of generalized random walks under differing moment bounds and degrees of independence. We also note that the above theorems can be used to prove the following result, which almost completely characterizes all moment bounds of the furthest distance traveled by a $k$-wise independent random walk with steps in $\pm 1$.

\begin{theorem} \label{Main}
    Let $q \ge p \ge 1,$ $k \ge 2$, and $n \ge 1$ be positive integers. If $X_1, \dots, X_n$ are unbiased $k$-wise independent random variables with $\BE[|X_i|^q] \le 1$ for all $i$, then 
\begin{equation} \label{Bound6}
    \BE\left[\sup_{1 \le i \le n} |X_1 + \dots + X_i|^p\right] = O\left(\sqrt{p \cdot n}\right)^p
\end{equation}
    whenever $k \ge 4, q \ge 2,$ and $2 \lfloor \frac{k}{2} \rfloor \ge p.$ If $k \le 3$ and $p \le 2 \le q,$ then 
\begin{equation} \label{Bound7}
\BE\left[\sup_{1 \le i \le n} |X_1 + \dots + X_i|^p\right] = O\left(\sqrt{n} \log n\right)^p.
\end{equation}
    The above bounds are tight for certain $\pm 1$ random variables, for all instances of $p, q, k$ and sufficiently large $n$.
    
    Finally, in all other cases, it is even possible that
\begin{equation} \label{Bound8}
\BE\left[|X_1 + \dots + X_n|^p\right] = \Omega(n)^{(p+1)/2},
\end{equation}
    making this case less interesting.
\end{theorem}

As the proof of Theorem \ref{Main} is relatively simple given the other results, and gives an overarching viewpoint of all the results, we prove it here, assuming Theorems \ref{Bound1}, \ref{Bound2}, \ref{Bound3}, \ref{Bound4}, and \ref{Bound5}. We defer the final case (i.e. Equation \eqref{Bound8}) to Appendix \ref{FinishFullTheorem}.

\begin{proof}
    First, consider the case $k \ge 4, q \ge 2,$ and $2 \lfloor \frac{k}{2} \rfloor \ge p.$ If $p \le 2,$ then note that $X_1, \dots, X_n$ are $4$-wise independent with $\BE[X_i^2]^{1/2} \le \BE[|X_i|^q]^{1/q} \le 1$ for all $i$. Thus, by Theorem \ref{Bound3}, we have $\BE[\sup (X_1+\dots+X_i)^2] = O(n),$ and $\BE[\sup |X_1+\dots+X_i|] \le \BE[\sup (X_1+\dots+X_i)^2]^{1/2} = O(\sqrt{n})$. If $p \ge 3,$ then let $r = 2 \lceil \frac{p}{2}\rceil,$ so $k \ge r \ge p$ and $r$ is even. Then, $X_1, \dots, X_n$ are $r$-wise independent, so $\BE[\sup |X_1+\dots+X_i|^p]^{1/p} \le \BE[\sup |X_1+\dots+X_i|^r]^{1/r} = O(\sqrt{r \cdot n}).$ However, $r = O(p),$ so $\BE[\sup |X_1+\dots+X_i|^p]^{1/p} = O(\sqrt{p \cdot n}).$ This proves the first case. We note that this inequality is sharp since even in the case where $X_1, \dots, X_n$ are i.i.d. uniformly random sign variables, $\BE[|X_1+\dots+X_i|^p] = \Theta(\sqrt{p \cdot n})^p$ by Khintchine's inequality \cite{Khintchine}.
    
    Next, consider the case where $k \le 3$ and $p \le 2 \le q.$ Then, $\BE[X_i^2]^{1/2} \le \BE[|X_i|^q]^{1/q} \le 1$ for all $i$, and since $k \ge 2,$ we have $X_1, \dots, X_n$ are pairwise independent. Thus, by Theorem \ref{Bound2}, we have that $\BE[\sup (X_1+\dots+X_i)^2] = O(n \log^2 n)$ and $\BE[\sup |X_1+\dots+X_i|] \le \BE[\sup (X_1+\dots+X_i)^2]^{1/2} = O(\sqrt{n} \log n).$ This proves the second case. We note that this inequality is sharp since by Theorem \ref{Bound1}, there exist $3$-wise independent $X_1, \dots, X_n$ such that each $X_i$ is a uniformly random sign variable, but $\BE[\sup (X_1+\dots+X_i)^2]^{1/2} \ge \BE[\sup |X_1+\dots+X_i|] = \Omega(\sqrt{n} \log n).$
\end{proof}

Finally, we note that to prove Theorem \ref{Bound1}, we create a complicated pairwise independent hash function, which suggests that standard pairwise independent hash functions do not have this property. Indeed, we show in Appendix \ref{Appendix} that for some types of hash functions constructed from Hadamard matrices or random linear functions from $\BZ/p\BZ \to \BZ/p\BZ,$ we have $\BE \sup_i |h_1 + \cdots + h_i| = O(\sqrt{n}).$

\subsection{Motivation and Relation to Previous Work}

The primary motivation of this paper comes from certain theorems that provide strong bounds for certain variants of $4$-wise independent random walks, which raised the question of whether any of these bounds can be extended to $2$-wise or $3$-wise independence. For example, Theorem 1 in \cite{BPTree} proves for any unbiased $4$-wise independent hash family $\mathcal{H}$, $\BE_{h \sim \mathcal{H}} \left(\sup_t \langle h, z^{(t)}\rangle\right) = O(\|z\|_2).$  This result generalizes a result from \cite{CountSieve} which proves the same result if $h$ is uniformly chosen from $\{-1, 1\}^n.$  \cite{BPTree} provides an algorithm that successfully finds all $\ell_2$ $\varepsilon$-heavy hitters in an insertion-only stream in $O(\varepsilon^{-2} \log \varepsilon^{-1} \log n)$ bits of space, in which the above result was crucial for analysis of a subroutine which attempts to find bit-by-bit the index of a single ``super-heavy'' heavy hitter if one exists. Theorem 1 in \cite{BPTree} also proved valuable for an algorithm for continuous monitoring of $\ell_p$ norms in insertion-only data streams \cite{LpContinuousMonitor}. Lemma 18 in \cite{Jarek} shows that even without bounded fourth moments, given $4$-wise independent random variables $X_1, \dots, X_n$, each with mean $0$ and finite variance,
\[\BP\left(\max\limits_{1 \le i \le n} |X_1+\dots+X_i| \ge \lambda\right) = O\left(\frac{n \cdot \max_i \BE[X_i^2]}{\lambda^2}\right).\]
This theorem was crucial in analyzing an algorithm tracking distinct elements that provides a $(1+\varepsilon)$-approximation with failure probability $\delta$ in $O(\varepsilon^{-2} \lg \delta^{-1} + \lg n)$ bits of space. Notice that Theorem \ref{Bound3} is stronger than the above equation and implies Kolmogorov's inequality up to a constant factor, but requires much weaker assumptions.

A natural follow-up question to the above theorems is whether $4$-wise independence is necessary, or whether more limited levels of independence such as $2$-wise or $3$-wise independence is allowable. Theorem \ref{Bound1} shows that $3$-wise independence does not suffice for any of the above results, because the random walk on a line case is strictly weaker than all of the above results. As a result, we know that the tracking sketches in \cite{BPTree, LpContinuousMonitor, Jarek} require $4$-wise independence.

However, the results given still have interesting extensions, such as to higher moments. For instance, Theorem \ref{Bound5} shows a stronger result than the one established in \cite{BPTree}, since it not only bounds the first moment of $\sup_t \langle h, z^{(t)}\rangle$ for $h$ sampled from a $4$-wise independent family of uniform $\pm 1$ variables but also bounds the $4$th moment tightly. 

Outside of random walks, $k$-wise independence for hash functions, which was first introduced in \cite{kWise}, has been crucial to numerous algorithms. Bounding the amount of independence required for analysis of algorithms has been studied in various contexts, often since $k$-wise independent hash families can be stored in low space but may provide equally adequate bounds as totally independent families. As further examples, the well-known AMS sketch \cite{AMS} is a streaming algorithm to estimate the $\ell_2$ norm of a vector $z$ to a factor of $1 \pm \varepsilon$ with high probability by multiplying the vector by a sketch matrix $\Pi \in \BR^{n \times (1/\varepsilon^2)}$ of $4$-wise independent random signs and using $\|\Pi z\|_2$ as an estimate for $\|z\|_2.$  It is known from \cite{LinearProbing, Private} that the accuracy of the AMS sketch can be much worse if $3$-wise independent random signs are used instead of $4$-wise independent random signs. If $z$ is given as an insertion stream, it is currently known that the AMS sketch with $8$-wise independent random signs can provide weak tracking \cite{BPTree}, meaning that $\BE \sup_t \left|\|\Pi z^{(t)}\|_2^2 - \|z^{(t)}\|_2^2\right| \le \varepsilon \|z\|_2^2$. This implies that the approximation of the $\ell_2$ norm with the $8$-wise independent AMS sketch is quite accurate at all times $t$. While one cannot perform weak tracking with $3$-wise independence of the AMS sketch, it is unknown for $4$-wise independence through $7$-wise independence whether the AMS sketch provides weak tracking. Finally, linear probing, a well-known implementation of hash tables, was shown to take $O(1)$ expected update time with any $5$-wise independent hash function \cite{Pagh} but was shown to take $\Theta(\lg n)$ expected update time for certain $4$-wise independent hash functions and $\Theta(\sqrt{n})$ expected update time for certain $2$-wise independent hash functions \cite{LinearProbing}.

Bounding the maximum distance traveled of a random walk has also been studied in probability theory independent of computer science applications, both when the steps are totally independent or $k$-wise independent. For example, Kolmogorov's inequality \cite{Kolmogorov} provides strong bounds for $\sup_t (X_1+\dots+X_t)$ for independent random variables $X_1, \dots, X_t$, as long as the second moments of $X_1, \dots, X_t$ are finite. \cite{RandomWalkBounded} constructed an infinite sequence $\{X_1, X_2, \dots\}$ of pairwise independent random variables taking on the values $\pm 1$ such that $\sup_t (X_1+\dots+X_t)$ is bounded almost surely, though the paper also proved that this phenomenon can never occur for $4$-wise independent variables taking on the values $\pm 1$. Finally, the supremum of a random walk with i.i.d. bounded random variable steps was studied in \cite{SupremumSumRVs}, which provided comparisons with the supremum of a Brownian motion random walk regardless of the random variable chosen for step size.

\subsection{Overview of Proof Ideas}
Here, we briefly outline some of the main ideas behind the proofs of our theorems.

The main goal in Section \ref{LowerBounds} is to establish Theorem \ref{Bound1}, i.e. construct a $3$-wise independent hash family $\mathcal{H}$ such that $\BE_{h \sim \mathcal{H}}\left[\sup_{1 \le i \le n} |h_1+\dots+h_i|\right] = \Omega(\sqrt{n} \lg n)$. To have $3$-wise independence, it will suffice to find a distribution $\mathcal{H}$ over $\{-1, 1\}^n$ so that $\BE_{h \sim \mathcal{H}}[h_ih_j] = 0$ for all $i \neq j$: it will be easy to show that if we replace $(h_1, \dots, h_n)$ with $(-h_1, \dots, -h_n)$ with $1/2$ probability, the hash function will be $3$-wise independent. In other words, we wish for the covariance matrix $\mathcal{M} = \BE[h^T h]$ to be the identity matrix $I_n$. The construction has two major steps.
\begin{enumerate}
    \item Create a hash function such that $\BE \sup_{1 \le i \le n} |h_1+\dots+h_i| = \Omega(\sqrt{n} \lg n)$ but rather than have $\BE[h_ih_j] = 0$ for all $i \neq j,$ have $\sum_{i \neq j} |\BE[h_ih_j]| = O(n),$ i.e. the cross terms in total aren't very large in absolute value (this hash function will be $\mathcal{H}_2$ in our proof). To do this, we first created $\mathcal{H}_1,$ which certain properties, most notably that $\BE[h_1 + \dots + h_n] = 0$ but $\BE[h_1 + \dots + h_{n/2}] = \Theta(\sqrt{n} \lg n)$, and rotated the hash family by a uniform index. The rotation allows many of the cross terms to average out, reducing the sum of their absolute values.
    \item Remove the cross terms. To do this, we make $\mathcal{H}$ a hash family where with some constant probability, we choose from $\mathcal{H}_2$ and with some probability, we choose some set of indices and pick a hash function such that $\BE[h_ih_j]$ will be the opposite sign of $\BE_{h \sim \mathcal{H}_2}[h_ih_j]$ for certain indices $i, j$, so that overall, $\BE[h_ih_j]$ will be $0$. Certain symmetry properties and most importantly the fact that $\sum_{i \neq j} |\BE_{h \sim \mathcal{H}_2} [h_ih_j]| = O(n)$ will allow for us to choose from $\mathcal{H}_2$ with constant probability, which means even for our final hash function,  $\BE \sup_{1 \le i \le n} |h_1+\dots+h_i| = \Omega(\sqrt{n} \lg n)$.
\end{enumerate}

The goal of Section \ref{PairwiseUpperBound} is to establish Theorem \ref{Bound2}, i.e. to show that if $\mathcal{M} = \BE[h h^T] = I_n,$ which is true for any pairwise independent hash function, then $\sup_{1 \le i \le n} |h_1+\dots+h_i|^2 = O(n \lg^2 n).$  To do this, we apply ideas based on the second moment method. We notice that for any matrix $A$, $\BE[h^T A h] = Tr(A),$ and thus, if we can find a matrix such that the trace of the matrix is small, but $h^T A h$ is reasonably large in comparison to $\sup_{1 \le i \le n} |h_1+\dots+h_i|^2,$ then $\BE[h^T A h]$ is small but is large in comparison to $\BE\left[\sup_{1 \le i \le n} |h_1+\dots+h_i|^2\right].$  If we assume that $n$ is a power of $2$, then the matrix that corresponds to the quadratic form 
\[h^T A h = \sum\limits_{r = 0}^{\lg n} \sum\limits_{i = 0}^{(n/2^r) - 1} (h_{i \cdot 2^r + 1} + \dots + h_{(i+1) \cdot 2^r})^2,\]
i.e. $h^T A h = h_1^2 + \dots + h_n^2 + (h_1 + h_2)^2 + \dots + (h_{n-1}+h_n)^2 + \dots + (h_1+\dots+h_n)^2$ can be shown to satisfy $Tr(A) = n \lg n$ and for any vector $x,$ $x^T A x \ge \frac{1}{\lg n} \cdot (x_1+\dots+x_i)^2$ for all $1 \le i \le n,$ not just in expectation. These conditions will happen to be sufficient for our goals. This method, in combination with Theorem \ref{Bound1}, will also allow us to prove an interesting matrix inequality, proven at the end of Section \ref{PairwiseUpperBound}. The method above actually generalizes to looking at $k$th moments of $k$-wise independent hash functions, as well as random walks corresponding to tracking insertion-only streams, and will allow us to prove Theorems \ref{Bound4} and \ref{Bound5}. However, these generalizations will also need the construction of $\varepsilon$-nets, which are explained in subsection \ref{Equation45}, or in \cite{JelaniChaining}.

We finally explain the ideas behind Theorem \ref{Bound3}, the generalization of Kolmogorov's inequality and Lemma 18 of \cite{Jarek}. We use ideas of chaining, such as in \cite{JelaniChaining}, and an idea of \cite{Jarek} that allows us to bound the minimum of $X_{i+1}+\dots+X_j$ and $X_{j+1}+\dots+X_k$ where $i < j < k,$ given $4$-wise independent functions $X_1, \dots, X_n$ with only bounded second moments. We combine these with another idea, that we can consider distances between $i$ and $j$ for $1 \le i < j \le n$ as $\BE[X_{i+1}^2 + \dots + X_j^2]$ and create our $\varepsilon$-nets using this distancing. 
The final additional idea is similar to the main idea of Theorem \ref{Bound2}: we write a degree $4$ function of $X_1, \dots, X_n$ based on the structure of the $\varepsilon$-nets and show that $\sup (X_1 + \dots + X_i)^2$ can be bounded by this degree $4$ function. This way, we get a bound on $\BE[\sup (X_1 + \dots + X_i)^2]$ rather than on the Chebyshev variant of it.
Together, these ideas allow for our chaining method to be quite effective, even if the $X_i$'s do not have bounded $4$th moments or if the $X_i$'s wildly differ in variance.

As a note, while \cite{BPTree} only bounds the first moment for the $4$-wise independent random walk, their approach can be easily modified to bound the $(k-\varepsilon)^{\text{th}}$ moment of a $k$-wise independent random walk for any $\varepsilon > 0$. Therefore, our main contribution is to improve this to the full $k^{\text{th}}$ moment.

\section{Lower Bounds for Pairwise Independence} \label{LowerBounds}

In this section, we prove Theorem \ref{Bound1}. In other words, we construct a $3$-wise independent family $\mathcal{H}$ such that the furthest distance traveled by the random walk is $\Omega(\sqrt{n} \lg n)$ in expected value.

To actually construct this counterexample, we proceed by a series of families and tweak each family accordingly to get to the next one, until we get the desired $\mathcal{H}$. We assume $n$ is a power of $4$: this assumption can be removed by replacing $n$ with the largest power of $4$ less than or equal to $n$.

First, we will need the following simple proposition about unbiased $k$-wise independent variables.

\begin{proposition} \label{Boolean}
    Let $\mathcal{H}$ be a distribution over $\{-1, 1\}^n$ such that $\BE_{h \sim \mathcal{H}} [h_i] = 0$ for all $i$. Then, $\mathcal{H}$ is a $k$-wise independent hash family if and only if for all $1 \le \ell \le k$ and all $1 \le i_1 < \cdots < i_\ell \le n,$ $\BE_{h \sim \mathcal{H}}[h_{i_1} \cdots h_{i_\ell}] = 0.$
\end{proposition}

\begin{proof}
    For the ``only if'' direction, if $h_1, \dots, h_n$ are $k$-wise independent for $h \sim H$, then $h_1, \dots, h_n$ are $\ell$-wise independent for all $\ell \le k.$ Therefore, $\BE[h_{i_1} \cdots h_{i_\ell}] = \BE[h_{i_1}] \cdots \BE[h_{i_\ell}] = 0.$
    
    For the ``if'' direction, let $1 \le i_1 < \cdots < i_k \le n$ and let $a_1, \dots, a_k \in \{-1, 1\}.$ Then, note that if $h_{i_j} = a_j$ for all $1 \le j \le k,$ $\prod (1 + a_jh_{i_j}) = 2^k$ since $1+a_j h_{i_j} = 2$ for all $j$. Otherwise, $1+a_j h_{i_j} = 0$ for some $j$, so $\prod (1 + a_jh_{i_j}) = 0$. Thus,
\[\BE_{h \sim \mathcal{H}}\left[\prod (1 + a_jh_{i_j})\right] = 2^k \cdot \BP_{h \sim \mathcal{H}}\left(h_{i_1} = a_1, \dots, h_{i_k} = a_k\right).\]
    However, if $\BE[h_{i_1} \cdots h_{i_\ell}] = 0$ for all $1 \le \ell \le k,$ then we have that $\BE\left[\prod (1 + a_jh_{i_j})\right] = 1$ by expanding $\prod (1+a_j h_{i_j})$ and applying linearity of expectation. Therefore, for all $a_1, \dots, a_k \in \{-1, 1\}$, we have $\BP\left(h_{i_1} = a_1, \dots, h_{i_k} = a_k\right) = \frac{1}{2^k},$ so $h_1, \dots, h_n$ are $k$-wise independent.
\end{proof}

We start by creating $\mathcal{H}_1$. First, split $[n]$ into blocks of size $\sqrt{n}$ so that $\{(c-1) \sqrt{n}+1, \dots, c \sqrt{n}\}$ form the $c$th block for each $1 \le c \le \sqrt{n}$. Also, define $\ell = \frac{\sqrt{n}}{2}$. Now, to pick a function $h$ from $\mathcal{H}_1$, choose the value of $h_i$ for each $1 \le i \le n$ independently, but if $i$ is in the $c$th block for some $1 \le c \le \ell,$ make $\BP[h_i = 1] = \frac{1}{2} + \frac{1}{2(\ell+1-c)}$ and if $i$ is in the $c$th block for some $\ell+1 \le c \le \sqrt{n},$ make $\BP[h_i = 1] = \frac{1}{2} - \frac{1}{2(c-\ell)}.$  This way, $\BE[h_i] = \frac{1}{\ell+1-c}$ if $i$ is in the $c$th block for $c \le \ell$ and $\BE[h_i] = -\frac{1}{c-\ell}$ if $i$ is in the $c$th block for $c > \ell.$

From now on, define $h_i$ to be periodic modulo $n,$ i.e. $h_{i+n} := h_i$ for all integers $i$. We first prove the following about $\mathcal{H}_1:$

\begin{lemma} \label{H1}
    Suppose that $1 \le i < j \le n.$  Suppose that $i$ is in block $c_1$ and $j$ is in block $c_2$, where $c_1$ and $c_2$ are not necessarily distinct. Define $r = \min(c_2-c_1, \sqrt{n}-(c_2-c_1)).$  Then, there is some constant $C_3$ such that
\[\sum\limits_{d = 0}^{\sqrt{n}-1} \BE_{h \sim \mathcal{H}_1} (h_{i + d \sqrt{n}} h_{j + d\sqrt{n}}) \le C_3 \cdot \frac{\lg (r+2)}{(r+1)^2}.\]
\end{lemma}

\begin{proof}
    For $1 \le c \le \sqrt{n},$ define $f_c$ to equal $\frac{1}{\ell+1-c}$ if $1 \le c \le \ell$ and to equal $-\frac{1}{c-\ell}$ if $\ell+1 \le c \le \sqrt{n}.$  In other words, $f_c = \BE[h_i]$ if $i$ is in the $c$th block. Furthermore, define $f$ to be periodic modulo $\sqrt{n},$ i.e. $f_c := f_{c+\sqrt{n}}$ for all integers $c$. Then,
\[\sum\limits_{d = 0}^{\sqrt{n}-1} \BE_{h \sim \mathcal{H}_1} (h_{i + d \sqrt{n}} h_{j + d\sqrt{n}}) = \sum\limits_{d = 0}^{\sqrt{n}-1} \BE_{h \sim \mathcal{H}_1} (h_{i + d \sqrt{n}}) \BE_{h \sim \mathcal{H}_1} (h_{j + d\sqrt{n}}) = \sum\limits_{d = 0}^{\sqrt{n}-1} f_{c_1+d} f_{c_2+d} = \sum\limits_{d = 1}^{\sqrt{n}} f_d f_{r+d}.\]

    Now, since $r \le \ell,$ if we assume $r \ge 1,$ this sum can be explicitly written as 
\begin{align*}
    &\hspace{0.5cm} 2 \cdot \sum\limits_{d = 1}^{\ell-r} \frac{1}{d(d+r)} - \sum\limits_{d = 1}^{r} \frac{1}{d(r+1-d)} - \sum\limits_{d = 1}^{r} \frac{1}{(n+1-d)(n+1-(r+1-d))}\\
    &\le 2\sum\limits_{d = 1}^{\infty} \frac{1}{d(d+r)} - \sum\limits_{d = 1}^{r} \frac{1}{d(r+1-d)} \\
    &= \frac{2}{r} \sum\limits_{d = 1}^{\infty} \left(\frac{1}{d} - \frac{1}{d+r}\right) - \frac{1}{r+1} \sum\limits_{d = 1}^{r} \left(\frac{1}{d} + \frac{1}{r+1-d}\right) \\
    &= \frac{2}{r} \left(\sum\limits_{d = 1}^r \frac{1}{d}\right) - \frac{2}{r+1} \left(\sum\limits_{d = 1}^{r} \frac{1}{d}\right) \\
    &= \frac{2}{r(r+1)} \left(\sum\limits_{d = 1}^{r} \frac{1}{d}\right) \le \frac{C_1 \lg (r+2)}{(r+1)^2}
\end{align*}
    for some constant $C_1.$
    If we assume $r = 0,$ then this sum can be explicitly written as
\[2 \cdot \sum\limits_{d = 1}^{\ell} \frac{1}{d^2} \le C_2 = \frac{(C_2) \cdot \lg (0+2)}{(0+1)^2}\]
    for some constant $C_2.$
    Therefore, setting $C_3 = \max(C_1, C_2)$ as our constant, we are done.
\end{proof}

    Next, we modify $\mathcal{H}_1$ to create $\mathcal{H}_2.$ To sample $h'$ from $\mathcal{H}_2,$ first choose $h \sim \mathcal{H}_1$ at random, and then choose an index $d$ between $0$ and $\sqrt{n}-1$ uniformly at random. Our chosen function $h'$ will then be the function that satisfies $h'_i = h_{i+d \cdot \sqrt{n}}$ for all $i$. We show the following about $\mathcal{H}_2$:

\begin{lemma} \label{H2}
The following three statements are true:
\begin{enumerate}[label=\alph*)]
    \item For all $i, j \in \BZ,$ $\BE_{h \sim \mathcal{H}_2} (h_ih_j) = \BE_{h \sim \mathcal{H}_2} (h_{i+\sqrt{n}}h_{j+\sqrt{n}}).$
    \item Suppose that $1 \le i, i', j, j' \le n$, where $i, i'$ are in blocks $c_1$, $j, j'$ are in blocks $c_2,$ and $i \neq j, i' \neq j'$. Then, $\BE_{h \sim \mathcal{H}_2} (h_ih_j) = \BE_{h \sim \mathcal{H}_2} (h_{i'}h_{j'}).$
    \item $\sum_{i \neq j} \left|\BE_{h \sim \mathcal{H}_2} h_i h_j\right| = O(n).$
\end{enumerate}
\end{lemma}

\begin{proof}
    Part a) is quite straightforward, since 
\[\BE_{h \sim \mathcal{H}_2}(h_ih_j) = \frac{1}{\sqrt{n}} \sum_{d = 0}^{\sqrt{n}-1} \BE_{h \sim \mathcal{H}_1} (h_{i+d\sqrt{n}}h_{j+d\sqrt{n}})\]
\[= \frac{1}{\sqrt{n}} \sum_{d = 0}^{\sqrt{n}-1} \BE_{h \sim \mathcal{H}_1} (h_{i+(d+1)\sqrt{n}}h_{j+(d+1)\sqrt{n}}) = \BE_{h \sim \mathcal{H}_2} (h_{i+\sqrt{n}}h_{j+\sqrt{n}})\]
    by periodicity of $h$ modulo $n$. 
    
    For part b), for all $d \in \BZ$, note that $i+d\sqrt{n}$ and $i'+d\sqrt{n}$ are in the same blocks, $j+d\sqrt{n}$ and $j'+d\sqrt{n}$ are in the same blocks, $i+d\sqrt{n} \neq j+d\sqrt{n}$ and thus $h_{i+d\sqrt{n}}, h_{j+d\sqrt{n}}$ are independent, and $i'+d\sqrt{n} \neq j'+d\sqrt{n}$ and thus $h_{i'+d\sqrt{n}}, h_{j'+d\sqrt{n}}$ are independent. Therefore, $\BE_{h \sim \mathcal{H}_1} (h_{i+d\sqrt{n}}h_{j+d\sqrt{n}}) = \BE_{h \sim \mathcal{H}_1} (h_{i'+d\sqrt{n}}h_{j'+d\sqrt{n}})$ for all $d$. Because of the way we constructed $\mathcal{H}_2,$ part b) is immediate from these observations.
    
    We use Lemma \ref{H1} to prove part c). First note that for all $i \neq j$,
\[\BE_{h \sim \mathcal{H}_2} (h_ih_j) = \frac{1}{\sqrt{n}} \sum\limits_{d = 0}^{\sqrt{n}-1} \BE_{h \sim \mathcal{H}_1} (h_{i+d\sqrt{n}} h_{j+d\sqrt{n}}) \le \frac{C_3 \lg(r+2)}{\sqrt{n} \cdot (r+1)^2},\]
    where $i$ is in block $c_1$, $j$ is in block $c_2$, and $r = \min(|c_1-c_2|, \sqrt{n}-|c_1-c_2|)$. Now, there are exactly $n (\sqrt{n}-1)$ pairs $(i, j)$ where $1 \le i, j \le n,$ $i \neq j$, and $r = 0.$  This is because we can choose from $\sqrt{n}$ blocks for the value of $c_1 = c_2,$ and then choose from $\sqrt{n}(\sqrt{n}-1)$ possible pairs $(i, j)$ in each block. For a fixed $0 < r < \ell,$ there are exactly $2 n^{3/2}$ pairs $(i, j)$, since there are $2\sqrt{n}$ choices for blocks $c_1$ and $c_2$ and $\sqrt{n}$ choices for each of $i$ and $j$ after that, for $r = \ell,$ there are exactly $n^{3/2}$ such pairs, since there are $2\sqrt{n}$ choices for blocks $c_1$ and $c_2$ and $\sqrt{n}$ choices for each of $i$ and $j$ after that, and finally we cannot have $r > \ell$. Therefore,
\[\sum\limits_{i \neq j} \max\left(0, \BE_{h \sim \mathcal{H}_2} (h_i h_j)\right) \le 2n^{3/2} \cdot \sum\limits_{r = 0}^{\ell} \frac{C_3 \lg (r+2)}{\sqrt{n} (r+1)^2} \le C_4 n\]
    for some constant $C_4,$ since $\sum \frac{\lg (r+2)}{(r+1)^2}$ is a convergent series.
    
    To finish, note that $|x| = 2 \cdot \max(0, x) - x,$ so
\[\sum\limits_{i \neq j} |\BE_{h \sim \mathcal{H}_2} (h_i h_j)| \le 2 \cdot C_4 n - \sum\limits_{i \neq j} \BE_{h \sim \mathcal{H}_2} (h_i h_j) \le (2 C_4 + 1)n,\]
    since
\[\sum\limits_{i \neq j} \BE_{h \sim \mathcal{H}_2} (h_i h_j) = \sum\limits_{i, j} \BE_{h \sim \mathcal{H}_2} (h_i h_j) - \sum\limits_{i} \BE_{h \sim \mathcal{H}_2} h_i^2 = \BE_{h \sim \mathcal{H}_2} (h_1+\dots+h_n)^2 - n \ge -n.\]

    Thus, setting $C_5 = 2C_4+1$ gets us our desired result.
\end{proof}

Next, we tweak $\mathcal{H}_2$ to create a new family $\mathcal{H}_3.$  First, notice that we can define $g_{c_1 c_2}$ for $1 \le c_1, c_2 \le \sqrt{n}$ to equal $\BE_{h \sim \mathcal{H}_2} (h_ih_j)$ for some $i$ in the $c_1$th block and $j$ in the $c_2$th block such that $i \neq j$. This is well defined by Lemma \ref{H2} b), and as $1 \le c_1, c_2 \le \sqrt{n},$ there always exist $i \neq j$ with $i$ in the $c_1$th block and $j$ in the $c_2$th block, as long as $n \ge 4$. Now, to sample from $\mathcal{H}_3,$ define $g = 1 + \sum_{c_1 < c_2} |g_{c_1c_2}|.$ Then, with probability $\frac{1}{g},$ we choose a hash function from $\mathcal{H}_2.$ For all $1 \le c_1 < c_2 \le \sqrt{n},$ with probability $\frac{|g_{c_1c_2}|}{g}$, we choose $h_i = 1$ for all $i$ in the $c_1$th bucketl; and if $g_{c_1c_2} \ge 0,$ we make $h_i = -1$ for all $i$ in the $c_2$th bucket but if $g_{c_1c_2} < 0,$ we make $h_i = 1$ for all $i$ in the $c_2$th bucket. Finally, if $i$ is not in either the $c_1$th or the $c_2$th bucket, we let $h_i$ be an independent random $\pm 1$ variable. We prove the following about $\mathcal{H}_3$:

\begin{lemma}
    If $i$ and $j$ are in different buckets, then $\BE_{h \sim \mathcal{H}_3} (h_ih_j) = 0.$  If $i, j$ are in the same bucket but $i \neq j$, then $0 \le \BE_{h \sim \mathcal{H}_3} (h_ih_j) \le C_5 \sqrt{n}$ for $C_5$ defined in Lemma \ref{H2}. Moreover, $\BE_{h \sim \mathcal{H}_3} (h_ih_j)$ is independent of the choice of $i, j$, i.e. if $i \neq j$ are in the same bucket and $i' \neq j'$ are in the same bucket, then $\BE_{h \sim \mathcal{H}_3} (h_ih_j) = \BE_{h \sim \mathcal{H}_3} (h_{i'}h_{j'})$.
\end{lemma}

\begin{proof}
    Assume WLOG that $i < j.$  If $i, j$ are in different buckets, then we compute $\BE_{h \sim \mathcal{H}_3}(h_i h_j)$ as follows. With probability $\frac{1}{g},$ we are choosing $h$ from $\mathcal{H}_2,$ and if $i$ is in the $c_1$th bucket and $j$ is in the $c_2$th bucket, then $\BE_{h \sim \mathcal{H}_2}(h_i h_j) = g_{c_1 c_2}.$  With probability $\frac{|g_{c_1c_2}|}{g}$ we have $h_i h_j = 1$ if $g_{c_1c_2} < 0$ and $h_i h_j = -1$ if $g_{c_1c_2} \ge 0.$  In all other scenarios, either $h_i$ or $h_j$ is a random sign completely independent of all other elements, which means that $\BE[h_ih_j] = 0.$  Therefore, the overall expected value of $h_ih_j$ equals $g_{c_1 c_2} \cdot \frac{1}{g} + \frac{|g_{c_1c_2}|}{g} \cdot \pm 1$ where the $\pm 1$ is positive if and only if $g_{c_1c_2} < 0,$ so the expected value is $0$.
    
    If $i, j$ are in the same bucket, then we can compute $\BE_{h \sim \mathcal{H}_3}(h_i h_j)$ as follows. With probability $\frac{1}{g},$ we are choosing $h$ from $\mathcal{H}_2,$ and if $i, j$ are in the $c$th bucket, then $\BE_{h \sim \mathcal{H}_2}(h_i h_j) = g_{c c}.$  For all $c' \neq c,$ there is a $\frac{|g_{cc'}|}{g}$ probability of everything in the $c$th block having the same sign and everything in the $c'$th block having the same sign. For the other cases, $i, j$ are independent Rademachers. Therefore, 
\[\BE_{h \sim \mathcal{H}_3} (h_ih_j) = \frac{g_{cc}}{g} + \sum_{c' \neq c} \frac{|g_{cc'}|}{g} = \frac{1}{g} \left(g_{cc} + \sum\limits_{c' \neq c} |g_{cc'}|\right).\]
    However, note that $g_{cc} \ge 0$ since $\BE_{h \sim \mathcal{H}_2}(h_ih_j) = \frac{1}{\sqrt{n}} \sum_d \BE_{h \sim \mathcal{H}_1} (h_{i+d\sqrt{n}} h_{j+d\sqrt{n}})$ and for all $d$, we have $\BE_{h \sim \mathcal{H}_1}(h_{i+d\sqrt{n}}h_{j+d\sqrt{n}}) \ge 0$ since $i+d\sqrt{n}, j+d\sqrt{n}$ are in the same block for all $d$. Furthermore, for all indices $c_1, c_2,$ $g_{c_1c_2} = g_{(c_1+1)(c_2+1)}$, where indices are taken modulo $\sqrt{n}$, by Lemma \ref{H2} a). Combining these gives 
\[\BE_{h \sim \mathcal{H}_3} (h_ih_j) = \frac{1}{\sqrt{n}} \cdot \left(\frac{1}{g} \cdot \left(\sum\limits_{c_1, c_2} |g_{c_1c_2}|\right)\right).\]
    However, we know that $g \ge 1$ and $\sum_{c_1, c_2} |g_{c_1 c_2}| \le C_5$ by the arguments of Lemma \ref{H2} c), so the lemma follows.
\end{proof}

    Next, we will modify $\mathcal{H}_3$ to create $\mathcal{H}_4,$ which will almost be our final hash family $\mathcal{H}$. First, let $C_6 = \frac{1}{g} \sum_{c_1, c_2} |g_{c_1c_2}|$. Note that $C_6$ is uniformly bounded by $C_5$, but the value of $C_6$ may depend on $n$. To sample from $\mathcal{H}_4$, with probability $p = \frac{1}{1+C_6(\sqrt{n}-1)/\sqrt{n}} \ge \frac{1}{1+C_6},$ choose $h$ from $\mathcal{H}_3.$ Else, for each block of $\sqrt{n}$ elements, choose a uniformly random subset of size $\ell$ from the block, and make the corresponding elements $1$ and the remaining elements $-1.$ 
    We show that $\BE_{h \sim \mathcal{H}_4} (h_ih_j) = 0$ for all $i \neq j.$  To see why, if $i$ and $j$ are in different blocks then $\BE_{h \sim \mathcal{H}_3} (h_ih_j) = 0$ and if we do not choose $h$ from $\mathcal{H}_3$, then $h_i$ and $h_j$ are independent. If $i, j$ are in the same block, then if we condition on choosing from $\mathcal{H}_3$, $\BE(h_ih_j) = \frac{C_6}{\sqrt{n}}.$  If we condition on not choosing from $\mathcal{H}_3,$ the probability of $i, j$ being the same sign is $\frac{(\sqrt{n}/2)-1}{\sqrt{n}-1} = \frac{\ell-1}{2\ell-1},$ meaning $\BE(h_ih_j) = -\frac{1}{\sqrt{n}-1}$. Therefore, $\BE_{h \sim \mathcal{H}_4} (h_ih_j) = p \cdot \frac{C_6}{\sqrt{n}} - (1-p) \cdot \frac{1}{\sqrt{n}-1} = 0.$
    
    Our final hash function will be $\mathcal{H}.$ To sample from $\mathcal{H},$ we sample $h \sim \mathcal{H}_4,$ but with probability $1/2,$ we replace $h_i$ with $-h_i$ for all $i$. We clearly will still have $\BE[h_ih_j] = 0$ for all $i \neq j.$ But due to our negation, we now have $\BE[h_i] = 0$ and $\BE[h_ih_jh_k] = 0$ for all $1 \le i, j, k \le n,$ since negating all $h_i$'s will also negate $h_ih_jh_k.$ Thus, by Proposition \ref{Boolean}, $\mathcal{H}$ is $3$-wise independent.
    
    To finish, it suffices to show that 
\[\BE_{h \sim \mathcal{H}} \left[\sup_{1 \le t \le n} |h_1+\dots+h_t|\right] = \Omega(\sqrt{n} \lg n).\]
    To check this, note that with probability $\frac{1}{2}$ we are sampling from $\mathcal{H}_4,$ so with probability at least $\frac{1}{2(1+C_6)}$ we are sampling from $\mathcal{H}_3,$ so we need to just verify that 
\[\BE_{h \sim \mathcal{H}_3} \left[\sup_{1 \le t \le n} |h_1+\dots+h_t|\right] = \Omega(\sqrt{n} \lg n).\]
    But for $\mathcal{H}_3,$ we are choosing something from $\mathcal{H}_2$ with probability $\frac{1}{g}$ but $g \le 1+C_5$ by the arguments of Lemma \ref{H2} c), so it suffices to verify that 
\[\BE_{h \sim \mathcal{H}_2} \left[\sup_{1 \le t \le n} |h_1+\dots+h_t|\right] = \Omega(\sqrt{n} \lg n).\]
    But for $\mathcal{H}_2,$ if we condition on the shifting index $d$, we know that 
\[\BE[h_{1 + d \sqrt{n}} + h_{2 + d \sqrt{n}} + \dots + h_{(d+\ell) \sqrt{n}}] \ge \sqrt{n} \left(1 + \dots + \frac{1}{\ell}\right) \ge C_7 \sqrt{n} \lg n\]
    for some $C_7,$ and likewise 
\[\BE[h_{1 + (d+\ell) \sqrt{n}} + h_{2 + (d+\ell) \sqrt{n}} + \dots + h_{(d+2\ell) \sqrt{n}}] \le \sqrt{n} \left(-1 - \dots - \frac{1}{\ell}\right) \le -C_7 \sqrt{n} \lg n,\]
    which means that regardless of whether $d \le \ell$ or $d > \ell,$
\[\BE_{h \sim \mathcal{H}_2} \left[\max\left(|h_1+\dots+h_{d \sqrt{n}}|, |h_1+\dots+h_{(d+\ell) \sqrt{n}}|\right)\right] \ge \frac{C_7}{2} \sqrt{n} \lg n\]
    by the triangle inequality. But for any $h \sim \mathcal{H}_2,$
\[\max\left(|h_1+\dots+h_{d \sqrt{n}}|, |h_1+\dots+h_{(d+\ell) \sqrt{n}}|\right) \le \sup\limits_{1 \le t \le n} (h_1+\dots+h_t),\]
    so the result follows by taking the expected value of both sides, which proves our upper bound is tight in the case of a random walk. Thus, we have proven Theorem \ref{Bound1}.
    
\section{Moment Bounds for Pairwise Independence} \label{PairwiseUpperBound}
We show that the bound established in Section \ref{LowerBounds} and the induced bound on the second moment are tight for the $2$-wise independent random walk case by proving Theorem \ref{Bound2}. While we prove a generalization of this result in Section \ref{UpperBounds}, the proof here is simpler. We will also assume $n$, the length of the random walk, is a power of $2$, as this assumption can be removed by replacing $n$ with the smallest power of $2$ greater than or equal to $n$.

To prove Theorem \ref{Bound2}, we first establish the following lemma:

\begin{lemma} \label{Matrix}
    Suppose that $A \in \BR^{n \times n}$ is a positive definite matrix such that $Tr(A) = d_1$ for some $d_1 > 0$. Also, suppose there is some $d_2 > 0$ such that for all vectors $x \in \BR^n$ and integers $1 \le i \le n,$ if $x_1 + \dots + x_i = 1,$ then $x^T A x \ge \frac{1}{d_2}$. Then, for all unbiased pairwise independent hash families $\mathcal{H}: [n] \to \{-1, 1\},$ 
\[\BE_{h \sim \mathcal{H}} \left(\sup\limits_{1 \le i \le n} (h_1+\dots+h_i)^2\right) \le d_1d_2.\]
\end{lemma}

\begin{proof}
    Note that $\BE_{h \sim \mathcal{H}} h_i^2 = 1$ for all $i$ and $\BE_{h \sim \mathcal{H}} (h_i h_j) = 0$ for all $i \neq j$. Therefore, 
\[\BE_{h \sim \mathcal{H}} (h^T A h) = \sum\limits_{1 \le i, j \le n} \BE_{h \sim \mathcal{H}} 
(h_i h_j A_{ij}) = \sum\limits_{1 \le i, j \le n} A_{ij} \left(\BE_{h \sim \mathcal{H}} (h_i h_j)\right)  = \sum\limits_{1 \le i \le n} A_{ii} = Tr(A) = d_1.\]

    However, for any $1 \le i \le n,$ for any $h \sim \mathcal{H},$ if $h_1+\dots+h_i \neq 0,$ then 
\[h^T A h \ge (h_1+\dots+h_i)^2 \cdot \frac{1}{d_2},\]
    since the vector $\frac{1}{h_1+\dots+h_i} \cdot h$ has its first $i$ components sum to $1$, so we can let this vector equal $x$ to get $x^T A x \ge \frac{1}{f(n)}.$  If $h_1+\dots+h_i = 0,$ then the above inequality is still true as $A$ is positive definite.
    
    Therefore,
\[h^T A h \ge \frac{1}{d_2} \cdot \sup\limits_{1 \le i \le n} (h_1+\dots+h_i)^2,\]
    which means that 
\[d_1 = \BE_{h \sim \mathcal{H}} (h^T A h) \ge \frac{1}{d_2} \cdot \BE_{h \in H} \left(\sup\limits_{1 \le i \le n} (h_1+\dots+h_i)^2\right),\]
    so we are done.
\end{proof}

\begin{lemma} \label{Matrix2}
    There exists a positive definite matrix $A \in \BR^{n \times n}$ such that $Tr(A) = n \lg n$ and for all $x \in \BR^n$ and $1 \le i \le n,$ if $x_1+\dots+x_i = 1,$ then $x^TAx \ge \frac{1}{\lg n}.$  This clearly implies Theorem \ref{Bound2}.
\end{lemma}

\begin{proof}
    Consider the matrix $A$ such that for all $1 \le i, j \le n,$ $A_{ij} = \lg n - k$ if $k$ is the smallest nonnegative integer such that $\lfloor \frac{i-1}{2^k} \rfloor = \lfloor \frac{j-1}{2^k} \rfloor.$  Alternatively, we can think of $A$ as the sum of all matrices $B^{ij}$, where $B^{ij}$ is a matrix such that $B^{ij}_{k \ell} = 1$ if $i \le k, \ell \le j$ and $0$ otherwise. However, we sum this not over all $1 \le i, j \le n$ but for $i = 2^r \cdot (s-1) + 1, j = 2^r \cdot s$ for $0 \le r \le \lg n - 1$ and $1 \le s \le 2^{\lg n-r}$. As an illustrative example, for $n = 8,$ $A$ equals
\[\left(\begin{matrix} 3 & 2 & 1 & 1 & 0 & 0 & 0 & 0 \\ 2 & 3 & 1 & 1 & 0 & 0 & 0 & 0 \\ 1 & 1 & 3 & 2 & 0 & 0 & 0 & 0 \\ 1 & 1 & 2 & 3 & 0 & 0 & 0 & 0 \\ 0 & 0 & 0 & 0 & 3 & 2 & 1 & 1 \\ 0 & 0 & 0 & 0 & 2 & 3 & 1 & 1 \\ 0 & 0 & 0 & 0 & 1 & 1 & 3 & 2 \\ 0 & 0 & 0 & 0 & 1 & 1 & 2 & 3 \end{matrix}\right) \]
\[= \left(\begin{matrix} 1 & 0 & 0 & 0 & 0 & 0 & 0 & 0 \\ 0 & 1 & 0 & 0 & 0 & 0 & 0 & 0 \\ 0 & 0 & 1 & 0 & 0 & 0 & 0 & 0 \\ 0 & 0 & 0 & 1 & 0 & 0 & 0 & 0 \\ 0 & 0 & 0 & 0 & 1 & 0 & 0 & 0 \\ 0 & 0 & 0 & 0 & 0 & 1 & 0 & 0 \\ 0 & 0 & 0 & 0 & 0 & 0 & 1 & 0 \\ 0 & 0 & 0 & 0 & 0 & 0 & 0 & 1 \end{matrix}\right) + \left(\begin{matrix} 1 & 1 & 0 & 0 & 0 & 0 & 0 & 0 \\ 1 & 1 & 0 & 0 & 0 & 0 & 0 & 0 \\ 0 & 0 & 1 & 1 & 0 & 0 & 0 & 0 \\ 0 & 0 & 1 & 1 & 0 & 0 & 0 & 0 \\ 0 & 0 & 0 & 0 & 1 & 1 & 0 & 0 \\ 0 & 0 & 0 & 0 & 1 & 1 & 0 & 0 \\ 0 & 0 & 0 & 0 & 0 & 0    & 1 & 1 \\ 0 & 0 & 0 & 0 & 0 & 0 & 1 & 1 \end{matrix}\right) + \left(\begin{matrix} 1 & 1 & 1 & 1 & 0 & 0 & 0 & 0 \\ 1 & 1 & 1 & 1 & 0 & 0 & 0 & 0 \\ 1 & 1 & 1 & 1 & 0 & 0 & 0 & 0 \\ 1 & 1 & 1 & 1 & 0 & 0 & 0 & 0 \\ 0 & 0 & 0 & 0 & 1 & 1 & 1 & 1 \\ 0 & 0 & 0 & 0 & 1 & 1 & 1 & 1 \\ 0 & 0 & 0 & 0 & 1 & 1 & 1 & 1 \\ 0 & 0 & 0 & 0 & 1 & 1 & 1 & 1 \end{matrix}\right).\]

It is easy to see that $Tr(A) = n \lg n,$ since $A_{ii} = \lg n$ for all $i$. For any $1 \le i < n$, define $i_0 = 0$ and for any $1 \le r \le \lg n,$ define $i_r = 2^{\lg n - r} \cdot \lfloor \frac{i}{2^{\lg n - r}} \rfloor.$  Then, for any $1 \le i < n,$ one can see that $i_{\lg n} = i$ and for any $1 \le i \le n,$ $A = B^{1 i_1} + B^{(i_1+1) i_2} + \dots + B^{(i_{lg n - 1}+1) i_{\lg n}} + C,$ where $C$ is some positive semidefinite matrix and we assume $B^{ij}$ is the $0$ matrix if $i = j+1,$ because $B^{1 i_1}$ and $B^{(i_{r-1} + 1) i_{r}}$ for all $1 \le r \le \lg n$ are verifiable as matrices in the summation of $A$. Therefore, if $x_1+\dots+x_i = 1,$
\[x^T A x \ge \sum\limits_{i = 1}^{r} x^T B^{(i_{r-1} + 1) i_{r}} x = (x_1+\dots+x_{i_1})^2 + (x_{i_1+1}+\dots+x_{i_2})^2 + \dots + (x_{i_{\lg n - 1}+1} + \dots + x_{i_{\lg n}})^2 \ge \frac{1}{\lg n},\]
    since $(x_1+\dots+x_{i_1}) + (x_{i_1+1}+\dots+x_{i_2}) + \dots + (x_{i_{\lg n-1}+1} + \dots + x_{i}) = 1$ and by Cauchy-Schwarz.
    
Finally, if $i = n,$ then $A = B^{1 (n/2)} + B^{(n/2+1) n} + C,$ where $C$ is some positive semidefinite matrix. Therefore, if $x_1+\dots+x_n = 1,$
\[x^T A x \ge x^T B^{1 (n/2)} x + x^T B^{(n/2+1) n} x = (x_1+\dots+x_{n/2})^2 + (x_{n/2+1} + \dots + x_n)^2 \ge \frac{1}{2} \ge \frac{1}{\lg n}. \qedhere\]
\end{proof}

As a final note, for any positive definite matrix $A$ and vector $v,$ the minimum value of $w^T A w$ over all $w$ such that $w^T v = 1$ is known to equal $(v^T A^{-1} v)^{-1}$. This can be checked with Lagrange Multipliers, since the Lagrangian $f(w, \lambda)$ of $f(w) = w^T A w$ subject to $w^T v = 1$ equals $w^T A w - \lambda(w^T v - 1),$ which is a convex function in $w$ and has its derivatives vanish on the hyperplane $w^Tv = 1$ when $\lambda = 2 (v^T A^{-1} v)^{-1}, w = \frac{\lambda}{2} (A^{-1} v)$ (See for example \cite{LagrangeMultipliers}, Chapter 5, for more details of Lagrange Multipliers). By Lemma \ref{Matrix} and Theorem \ref{Bound1}, we have the following corollary:

\begin{corollary}
    For all positive definite $A,$ if we define $v^i$ as the vector with first $i$ components $1$ and last $n-i$ components $0$,
\[Tr(A) \cdot \max\limits_{1 \le i \le n} ((v^i)^T A^{-1} v^i) = \Omega(n \lg^2 n)\]
    and this bound is tight for the matrix of Lemma \ref{Matrix2}.
\end{corollary}

\begin{proof}
    If the first part were not true, then there would be matrices $A_n$ such that $Tr(A) = d_1,$ $w^T A w = \frac{1}{d_2}$ where $w^T v^i = 0$ for some $i$, and $d_1d_2 = o(n \lg^2 n).$  However, this would mean by Lemma \ref{Matrix} that for all pairwise independent $\mathcal{H},$
\[\BE_{h \sim \mathcal{H}} \left(\sup\limits_{1 \le i \le n} (h_1+\dots+h_i)^2\right) \le d_1d_2 = o(n \lg^2 n),\]
    contradicting Theorem \ref{Bound1}. The second part is immediate by the analysis of Lemma \ref{Matrix2}.
\end{proof}

\section{Generalized Upper Bounds} \label{UpperBounds}

In this section, our goal is to prove Theorems \ref{Bound3}, \ref{Bound4}, and \ref{Bound5} of Section \ref{MainResults}.

First, we prove a preliminary inequality, which will be useful for proving both Theorem \ref{Bound3} and Theorem \ref{Bound5}.

\begin{lemma} \label{EasyInequality}
    There exists some constant $c > 0$ such that for all real numbers $x_1, \dots, x_m,$ and all even $k \in \BN$,
\[\sum\limits_{r = 1}^{m} 2^{k \cdot r/8} x_r^k \ge (c(x_1 + \dots + x_m))^k.\]
\end{lemma}

\begin{proof}
    Assume WLOG that $x_1 + \dots + x_m = 1.$ Let $x_1' \ge x_2' \ge \dots \ge x_m'$ be real numbers such that $x_1' + \dots + x_m' = 1$ and $x_1' > \dots > x_m'$ are in a geometric series with common ratio $2^{-1/8}.$  Then, there exists $i$ such that $x_i \ge x_i'$.  But note that $(x_r')^k 2^{k \cdot r/8}$ are equal for all $r$ because of our geometric series, and equals $2^{k/8} \cdot (x_1')^k \ge (2^{1/8}-1)^k$. Thus, 
\[\sum\limits_{r = 1}^{m} 2^{k \cdot r/8} x_r^k \ge 2^{k \cdot i/8} (x_i')^k = (2^{1/8}-1)^k,\]
    so we are done.
\end{proof}

\subsection{Proof of Theorem \ref{Bound3}}

In this subsection, we prove Theorem \ref{Bound3}, which generalizes Kolmogorov's inequality \cite{Kolmogorov} to $4$-wise independent random variables $X_1, \dots, X_n$.

\begin{proof}[Proof of Theorem \ref{Bound3}]
    Assume WLOG that $\sum \BE[X_i^2] = 1$ (by scaling all of the variables) and $\BE[X_i^2] > 0$ for all $i$, i.e. none of the variables are almost surely $0$. Also, define $S_i = X_1 + \dots + X_i$ and $T_i = \BE[(X_1 + \dots + X_i)^2] = \BE[X_1^2 + \dots + X_i^2]$ for $0 \le i \le n.$  Note that $T_0 = 0$ and $T_n = 1.$

    We proceed by constructing a series of nested intervals $[a_{r, s}, b_{r, s}]$. We construct $a_{r, s}$ and $b_{r, s}$ for $0 \le r \le d = \left\lceil \max_i \lg \left(\BE[X_i^2]^{-1}\right)\right\rceil$ and $1 \le s \le 2^r,$ as integers between $0$ and $n$, inclusive.  First define $a_{0, 1} = 0$ and $b_{0, 1} = n.$  Next, we inductively define $a_{r, s}, b_{r, s}$.  Define $a_{r+1, 2s-1} := a_{r, s}$ and $b_{r+1, 2s} := b_{r, s}.$  Then, if there exists any index $a_{r, s} \le t \le b_{r, s}$ such that 
\[\left|T_{t} - T_{a_{r, s}}\right| = 0.5 \cdot \left|T_{b_{r, s}} - T_{a_{r, s}}\right|,\]
    let $a_{r+1, 2s} = b_{r+1, 2s-1} = t$.  Else, define $b_{r+1, 2s-1}$ to be the largest index $t \ge a_{r, s}$ such that 
\[\left|T_{t} - T_{a_{r, s}}\right| < 0.5 \cdot \left|T_{b_{r, s}} - T_{a_{r, s}}\right|\]
    and similarly define $a_{r+1, 2s}$ to be the smallest index $t \le b_{r, s}$ such that
\[\left|T_{t} - T_{a_{r, s}}\right| > 0.5 \cdot \left|T_{b_{r, s}} - T_{a_{r, s}}\right|.\]
    Note that in this case, $a_{r, 2s} = b_{r, 2s-1}+1.$
    
    It is clear that intervals are all nested in each other and for every fixed $r$, all integers between $0$ and $n$, inclusive, are in an interval $[a_{r, s}, b_{r, s}]$ for some $s$ (possibly at an endpoint).  Also, we always have $a_{r, 0} \le b_{r, 0} \le a_{r, 1} \le \dots \le b_{r, 2^r},$ and any interval $[a_{r, s}, b_{r, s}]$ satisfies $T_{b_{r, s}} - T_{a_{r, s}} \le 2^{-r}.$  This implies that since $d = \left\lceil \max_i \lg \left(\BE[X_i^2]^{-1}\right)\right\rceil,$ every integer equals $a_{d, s} = b_{d, s}$ for some $s.$
    
    We now consider the expression $f(X_1, \dots, X_n)$, defined to equal
\[\sum\limits_{r = 1}^{d} 2^{r/2} \cdot \sum\limits_{s = 1}^{2^{r-1}} \left((S_{b_{r, 2s}} - S_{a_{r, 2s}})^2 (S_{a_{r, 2s}} - S_{a_{r, 2s-1}})^2 + (S_{b_{r, 2s}} - S_{b_{r, 2s-1}})^2 (S_{b_{r, 2s-1}} - S_{a_{r, 2s-1}})^2\right),\]
    where we recall that $S_i = X_1 + \dots + X_i$.

We first prove the following lemma.
\begin{lemma} \label{KolmogorovLem1}
    $\BE\left[f(X_1, \dots, X_n)\right] \le 10.$
\end{lemma}

\begin{proof}
    By $4$-wise independence, we have that
\begin{align*}
\BE\left[(S_{b_{r, 2s}} - S_{a_{r, 2s}})^2 (S_{a_{r, 2s}} - S_{a_{r, 2s-1}})^2\right] &= \BE\left[(S_{b_{r, 2s}} - S_{a_{r, 2s}})^2\right] \cdot \BE\left[(S_{a_{r, 2s}} - S_{a_{r, 2s-1}})^2\right] \\
&= (T_{b_{r, 2s}} - T_{a_{r, 2s}}) \cdot (T_{a_{r, 2s}} - T_{a_{r, 2s-1}}) \\
&\le \left(T_{b_{r-1, s}} - T_{a_{r-1, s}}\right)^2 \\
&\le 2^{-2(r-1)} = 4 \cdot 2^{-2r}.
\end{align*}
    By the same argument, we also have that
\[\BE\left[(S_{b_{r, 2s}} - S_{b_{r, 2s-1}})^2 (S_{b_{r, 2s-1}} - S_{a_{r, 2s-1}})^2\right] \le 4 \cdot 2^{-2r}.\]
    
    Therefore, we get that
\begin{align*}
    \BE\left[f(X_1, \dots, X_n)\right] &\le \sum\limits_{r = 1}^{d} 2^{r/2} \cdot \left(\sum\limits_{s = 1}^{2^{r-1}} 8 \cdot 2^{-2r}\right) \\
    &\le \sum\limits_{r = 1}^{d} 2^{r/2} \cdot \left(4 \cdot 2^{-r}\right) \\
    &\le 4 \cdot \sum\limits_{r = 1}^{\infty} 2^{-r/2} = 4(\sqrt{2}+1),
\end{align*}
    which is at most $10$.
\end{proof}

Next, we show the following:

\begin{lemma} \label{KolmogorovLem2}
    There exists some constant $c > 0$ such that for all $X_1, \dots, X_n \in \BR$ and all $1 \le i \le n,$
\[c \cdot \min\left((X_1 + \dots + X_i)^4, (X_{i+1} + \dots + X_n)^4\right) \le f(X_1, \dots, X_n).\]
\end{lemma}

\begin{proof}
    Define $i_0, i_1, \dots, i_d$ as follows. First, let $i_d = i,$ which we note must equal either $a_{d, s}$ or $b_{d, s}$ for some $s$. Next, we inductively define $i_r$ given $i_{r+1}$ to be either $a_{r, s'}$ or $b_{r, s'}$ for some $s'$. 
    
    If $i_{r+1} = a_{r+1, s}$, where $s$ is chosen as small as possible, then if $s$ is odd, then $i_{r} := i_{r+1} = a_{r, (s+1)/2}.$ If $s$ is even, we choose $i_r$ to either equal $b_{r+1, s} = b_{r, s/2}$ or $a_{r+1, s-1} = a_{r, s/2}$. If $|S_{i_{r+1}} - S_{a_{r, s/2}}| \le |S_{i_{r+1}} - S_{b_{r, s/2}}|,$ then we choose $i_r := a_{r, s/2},$ and otherwise, we choose $i_r := b_{r, s/2}.$ 
    
    Else, if $i_{r+1} = b_{r+1, s}$, where $s$ is chosen as small as possible, then if $s$ is even, then $i_r := i_{r+1} = b_{r, s/2}.$ If $s$ is odd, we choose $i_r$ to be either $a_{r+1, s} = a_{r, (s+1)/2}$ or $b_{r+1, s+1} = b_{r, (s+1)/2}$. If $|S_{i_{r+1}} - S_{a_{r, (s+1)/2}}| \le |S_{i_{r+1}} - S_{b_{r, (s+1)/2}}|,$ then we choose $i_r := a_{r, (s+1)/2},$ and otherwise, we choose $i_r := b_{r, (s+1)/2}.$
    
    It is clear that for all $1 \le r \le d,$
\[(S_{i_r}- S_{i_{r-1}})^4 \le \sum\limits_{s = 1}^{2^{r-1}} \left((S_{b_{r, 2s}} - S_{a_{r, 2s}})^2 (S_{a_{r, 2s}} - S_{a_{r, 2s-1}})^2 + (S_{b_{r, 2s}} - S_{b_{r, 2s-1}})^2 (S_{b_{r, 2s-1}} - S_{a_{r, 2s-1}})^2\right),\]
    since we have that $(S_{i_r}- S_{i_{r-1}})^4$ is even smaller than some individual term in the sum on the right. Therefore, this implies that
\[\sum\limits_{r = 1}^{d} 2^{r/2} \cdot (S_{i_r} - S_{i_{r-1}})^4 \le f(X_1, \dots, X_n).\]
    But by Lemma \ref{EasyInequality}, we have that
\[\sum\limits_{r = 1}^{d} 2^{r/2} \cdot (S_{i_r} - S_{i_{r-1}})^4 \ge c \cdot (S_{i_d} - S_{i_0})^4 = c \cdot (S_i - S_{i_0})^4.\]
    But $i_0$ must equal $0$ or $n$, as we know that $i_0$ equals $a_{0, s}$ or $b_{0, s}$ for some $s$. Thus, we are done.
\end{proof}
    
    We are now close to completing the proof of the theorem. First, by Lemma \ref{KolmogorovLem1} and by taking the maximum over $i$ in Lemma \ref{KolmogorovLem2}, we have that there exists some constant $C$ such that
\[\BE\left[\max\limits_{1 \le i \le n} \left(\min\left((X_1 + \dots + X_i)^4, (X_{i+1} + \dots + X_n)^4\right)\right)\right] \le C.\]
    Noting that $\max (a_1^2, \dots, a_n^2) = (\max(a_1, \dots, a_n))^2$ and $\min(b_1^2, \dots, b_n^2) = (\min(b_1, \dots, b_n))^2$ whenever $a_1, \dots, a_n, b_1, \dots, b_n$ are nonnegative, we have by the Cauchy-Schwarz inequality that
\[\BE\left[\max\limits_{1 \le i \le n} \left(\min\left((X_1 + \dots + X_i)^2, (X_{i+1} + \dots + X_n)^2\right)\right)\right] \le \sqrt{C}.\]

    Now, it is straightforward to verify that $a^2 \ge \frac{a^2}{2} - (a+b)^2$ and $b^2 \ge \frac{a^2}{2} - (a+b)^2$ for all $a, b \in \BR,$ so $\min(a^2, b^2) \ge \frac{a^2}{2} - (a+b)^2$. By looking at $a = X_1 + \dots + X_i$ and $b = X_{i+1} + \dots + X_n,$ we have that
\begin{align*}
\max\limits_{1 \le i \le n} \left(\min\left((X_1 + \dots + X_i)^2, (X_{i+1} + \dots + X_n)^2\right)\right) &\le \max\limits_{1 \le i \le n} \left(\frac{(X_1 + \dots + X_i)^2}{2} - (X_1 + \dots + X_n)^2\right) \\
&= -(X_1 + \dots + X_n)^2 + \frac{1}{2} \max\limits_{1 \le i \le n} (X_1 + \dots + X_i)^2.
\end{align*}
    Therefore, since $T_n = \BE[(X_1 + \dots + X_n)^2] = 1,$ we in fact have that
\[\BE\left[\max\limits_{1 \le i \le n} (X_1 + \dots + X_n)^2\right] \le 2 (\sqrt{C} + 1) = O(1). \qedhere\]
\end{proof}

\subsection{Proof of Theorems \ref{Bound4} and \ref{Bound5}} \label{Equation45}

Before we prove Theorems \ref{Bound4} and \ref{Bound5}, we construct $2^{-r/2}$-nets for $0 \le r \le 2 \lg m + 1$ in a very similar way as in Theorem 1 in \cite{BPTree}. We define an $\varepsilon$-\textit{net} to be a finite set of points $a_{r, 0}, a_{r, 1}, \dots, a_{r, d_r}$ such that for every $z^{(t)},$ $\|z^{(t)}-a_{r, s}\|_2 \le \varepsilon \|z\|_2$ for some $0 \le s \le d_r.$ The constructions are defined identically for both Theorem \ref{Bound4} and Theorem \ref{Bound5}. Define $a_{0, 0} := z^{(0)}$ as the only element of the $2^{-0/2} = 1$-net. For $r \ge 1,$ define $a_{r, 0} = z^{(0)},$ and given $a_{r, s} = z^{(t_1)}$ then define $a_{r, s+1}$ as the smallest $t > t_1$ such that 
\[\|z^{(t)}-z^{(t_1)}\|_2 > 2^{-r/2} \cdot \|z\|_2^2,\]
unless such a $t$ does not exist, in which case let $s = d_r$ and do not define $a_{r, s'}$ for any $s' > s$.

We define the set $A_r = \{a_{r, s}: 0 \le s \le d_r\}.$  The following is directly true from our construction:
\begin{proposition} \label{EpsilonNet}
    For any $0 \le t \le m$ and fixed $r$, if $t_1 \le t$ is the largest $t_1$ such that $z^{(t_1)} = a_{r, s}$ for some $s$, then $\|z^{(t)}-z^{(t_1)}\|_2 \le 2^{-r/2} \cdot \|z\|_2.$  Consequently, $A_r = \{a_{r, 0}, \dots, a_{r, d_r}\}$ is a $2^{-r/2}$-net.
\end{proposition}

The above proposition implies the following:

\begin{proposition}
    For all $1 \le t \le m,$ $z^{(t)} = a_{2 \lg m + 1, s}$ for some $s.$
\end{proposition}

\begin{proof}
    Let $t_1$ be the largest integer at most $t$ such that $z^{(t_1)} = a_{2 \lg m + 1, s}$ for some $s.$  Then, $\|z^{(t)}-a_{2 \lg m + 1, s}\|_2^2 \le 2^{-(2 \lg m + 1)} \cdot \|z\|_2^2 < 1,$ which is clearly impossible unless $z^{(t)} = a_{2 \lg m + 1, s}.$
\end{proof}

Next, to prove Equations \eqref{Bound4} and \eqref{Bound5}, we will need the Marcinkiewicz–Zygmund inequality (see for example \cite{KhintchineGeneral}), which is a generalization of Khintchine's inequality (see for example \cite{Khintchine}):

\begin{theorem} \label{KhintchineGeneral}
    For any even $k \ge 2,$ there exists a constant $B_k$ only depending on $k$ such that for any fixed vector $v$ and totally independent and unbiased random variables $\vec{Y} = (Y_1, \dots, Y_n),$
\[\BE\left[\left(\sum\limits_{i = 1}^{n} Y_i\right)^k\right] \le B_k \BE\left[\left(\sum\limits_{i = 1}^{n} Y_i^2\right)^{k/2}\right].\]
    Moreover, $B_k$ grows as $\Theta(k)^{k/2}$.
\end{theorem}

This implies the following result:

\begin{proposition} \label{KhintchineCorollary}
    For any even $k \ge 2$ and vector $v \in \BR^n$, there exists a $B_k = \Theta(k)^{k/2}$ such that for any $k$-wise independent unbiased random variables $X_1, \dots, X_n$ with $\BE[X_i^k] \le 1$ for all $i$, and $\vec{X} := (X_1, \dots, X_n),$
\[\BE\left[\langle v, \vec{X}\rangle\right]^k = \BE\left[\left(\sum\limits_{i = 1}^{n} v_i X_i\right)^k\right] \le B_k \|v\|_2^k.\]
\end{proposition}

\begin{proof}
    Since the expected value of $\left(\sum v_iX_i\right)^k$ is only dependent on $k$-wise independence, we can assume that the $X_i$'s are totally independent but have the same marginal distribution. This implies
\[\BE\left[\left(\sum\limits_{i = 1}^{n} v_i X_i\right)^k\right] \le B_k \BE\left[\left(\sum\limits_{i = 1}^{n} v_i^2 X_i^2\right)^{k/2}\right]\]
    by Theorem \ref{KhintchineGeneral}. However, we know that $\BE[X_i^{2d}] \le 1$ for all $i$ and all $1 \le d \le k/2,$ since $\BE[X_i^k] \le 1$ and $\BE[X_i^{2d}]^{k/d} \le \BE[X_i^k]$ by Jensen's inequality, so simply expanding and using independence and linearity of expectation gets us the desired result.
\end{proof}

We now prove Theorems \ref{Bound4} and \ref{Bound5}.

\begin{proof}[Proof of Theorem \ref{Bound4}]
    For $r \ge 1$ and $s,$ suppose $a_{r, s} = z^{(t)}$ and $t_1 \le t$ is the largest index such that $z^{(t_1)} \in A_{r-1}.$  Then, define $f(s, t)$ to be the index $s'$ such that $z^{(t_1)} = a_{r-1, s'}.$  Consider the quadratic form 
\[\sum\limits_{r = 1}^{2 \lg m + 1} \sum\limits_{s = 0}^{d_r} \langle (a_{r, s}-a_{r-1, f(r, s)}), \vec{X} \rangle^2.\]
    By Proposition \ref{EpsilonNet}, $\|a_{r, s}-a_{r-1, f(r, s)}\|_2 \le 2^{-(r-1)/2} \cdot \|z\|_2$. Thus, by Proposition \ref{KhintchineCorollary}, we get the expected value of the quadratic form equals
\[\sum\limits_{r = 1}^{2 \lg m + 1} \sum\limits_{s = 0}^{d_r} \BE[\langle (a_{r, 2s+1}-a_{r, 2s}), \vec{X} \rangle^2] \le B_2 \sum\limits_{r = 1}^{2 \lg m + 1} \sum\limits_{s = 0}^{d_r} \|a_{r, 2s+1}-a_{r, 2s}\|_2^2\]
\[\le B_2 \sum\limits_{r = 1}^{2 \lg m + 1} \left(2^r \cdot 2^{-(r-1)}\|z\|_2^2\right) \le 2B_2(2 \lg m + 1)(\|z\|_2^2).\]
Here, I am using the fact that an $\varepsilon$-net has size at most $\varepsilon^{-2}$, which is easy to see since $z^{(0)}, \dots, z^{(m)}$ is tracking an insertion stream (it is proven, for example, in Theorem 1 of \cite{BPTree}), and thus $d_r \le 2^r$.

Now, for any $0 \le i \le n,$ consider $z^{(i)}$ and let $z^{(i)} = a_{2 \lg m + 1, s}.$  Then, define $s_r = s$ if $r = 2 \lg m + 1$ and $s_{r-1} = f(r, s_r)$ for $1 \le r \le 2 \lg m + 1.$  Note that $s_0 = 0$ and for any $r \ge 1,$ if $a_{r, s_r} \in A_{r-1},$ then $a_{r, s_r} = a_{r-1, s_{r-1}}.$  Thus, each $\langle (a_{r, s_r}-a_{r-1, s_{r-1}}), \vec{X}\rangle^2$ for $1 \le 2 \lg m + 1$ is either $0$ (because $a_{r, s_r}-a_{r-1, s_{r-1}} = 0$) or is a summand in our quadratic form. Therefore, 
\[\sum\limits_{r = 1}^{2 \lg m + 1} \sum\limits_{s = 0}^{d_r} \langle (a_{r, s}-a_{r, f(r, s)}), \vec{X} \rangle^2 \ge \sum\limits_{r = 1}^{2 \lg m + 1} \langle (a_{r, s_r}-a_{r-1, s_{r-1}}), \vec{X}\rangle^2 \ge \frac{1}{2 \lg m + 1} \cdot \langle z^{(i)}, \vec{X}\rangle^2,\]
with the last inequality true since $a_{2 \lg m + 1, s_{2 \lg m + 1}} = z^{(i)}$, $a_{0, s_0} = z^{(0)},$ and by the Cauchy-Schwarz inequality. As this is true for all $i$, taking the supremum over $i$ and then expected values gives us
\[2B_2(2 \lg m + 1)(\|z\|_2^2) \ge \BE\left[\sum\limits_{r = 1}^{2 \lg m + 1} \sum\limits_{s = 0}^{d_r} \langle (a_{r, 2s+1}-a_{r, 2s}), \vec{X} \rangle^2\right] \ge \frac{1}{2 \lg m + 1} \cdot \sup_i \BE\left[\langle z^{(i)}, \vec{X}\rangle^2\right],\]
and therefore,
\[\BE\left[\langle z^{(i)}, \vec{X}\rangle^2\right] = O\left(\|z\|_2^2 \cdot \lg^2 m\right). \qedhere\]
\end{proof}

\begin{proof}[Proof of Theorem \ref{Bound5}]
    Consider the form 
\[\sum\limits_{r = 1}^{2 \lg m + 1} 2^{k \cdot r/8} \sum\limits_{s = 0}^{d_r} \langle (a_{r, s}-a_{r-1, f(r, s)}), \vec{X} \rangle^k,\]
    with $f(r, s)$ defined as in the proof of Equation \eqref{Bound4}. We again note that by Proposition \ref{EpsilonNet}, $\|a_{r, s}-a_{r-1, f(r, s)}\|_2 \le 2^{-(r-1)/2} \cdot \|z\|_2$. Thus, by Proposition \ref{KhintchineCorollary}, we get the expected value of the form equals
\[\sum\limits_{r = 1}^{2 \lg m + 1} 2^{k \cdot r/8} \sum\limits_{s = 0}^{d_r} \BE[\langle (a_{r, s}-a_{r-1, f(r, s)}), \vec{X} \rangle^k] \le B_k \sum\limits_{r = 1}^{2 \lg m + 1} 2^{k \cdot r/8} \sum\limits_{s = 0}^{d_r} \|a_{r, s}-a_{r-1, f(r, s)}\|_2^k\]
\[\le B_k \sum\limits_{r = 0}^{2 \lg m + 1} 2^{k \cdot r/8} \cdot 2^r \cdot 2^{-(r-1)k/2} \|z\|_2^k \le B_k \cdot 2^{k/2} \sum\limits_{r = 0}^\infty 2^{r(8-3k)/8} \|z\|_2^k = O(2^{k/2} B_k) \|z\|_2^k,\]
since $k \ge 4.$  Again, I am using the fact that $d_r \le 2^r$ as an $\varepsilon$-net has size at most $\varepsilon^{-2}.$

Now, for any $0 \le i \le n,$ suppose $s$ satisfies $z^{(i)} = a_{2 \lg m + 1, s}.$  define $s_r = s$ if $r = 2 \lg m + 1$ and $s_{r-1} = f(r, s_r)$ for $1 \le r \le 2 \lg m + 1.$  Then, similarly to in the proof of Equation \eqref{Bound4},
\[\sum\limits_{r = 1}^{2 \lg m + 1} 2^{k \cdot r/8} \sum\limits_{s = 0}^{d_r} \langle (a_{r, s}-a_{r-1, f(r, s)}), \vec{X} \rangle^k \ge \sum\limits_{r = 1}^{2 \lg m + 1} 2^{k \cdot r/8} \langle (a_{r, s_r}-a_{r-1, s_{r-1}}), \vec{X}\rangle^k \ge (1-2^{-1/8})^k \cdot \langle z^{(i)}, \vec{X}\rangle^k.\]
The last inequality follows from Lemma \ref{EasyInequality}, by setting $x_r = \langle (a_{r, s}-a_{r-1, f(r, s)}), \vec{X}\rangle.$


As this is true for all $i$, we can take the supremum over $i$ and then take expected values to get
\[2^{k/2} B_k \cdot \|z\|_2^k \gtrsim \BE\left[\sum\limits_{r = 1}^{2 \lg m + 1} 2^{r/2} \sum\limits_{s = 0}^{d_r} \langle (a_{r, s}-a_{r-1, f(r, s)}), \vec{X} \rangle^k\right] \ge c^k \cdot \sup_i \BE\left[\langle z^{(i)}, \vec{X}\rangle^k\right]\]
for some fixed constant $c > 0.$ Therefore, for a fixed $k$,
\[\BE \sup\limits_{i} \left[\langle z^{(i)}, \vec{X}\rangle^k\right] = O\left(\|z\|_2^k\right),\]
and for variable $k$, as $B_k = \Theta(k)^{k/2}$, we have
\[\BE \sup\limits_{i} \left[\langle z^{(i)}, \vec{X}\rangle^k\right] = O\left(k^{1/2} \cdot ||z||_2\right)^k. \qedhere\]
\end{proof}

\section*{Acknowledgments}
This research was funded by the PRISE program at Harvard and the Herchel Smith Fellowship. I would like to thank Prof. Jelani Nelson for advising this work, as well as for problem suggestions, forwarding me many papers from the literature, and providing helpful feedback on my writeup. Finally, I would like to thank Prof. Scott Sheffield for explaining the Skorokhod embedding theorem and its application to the generalization of Kolmogorov's inequality.

\appendix
\section{Bounds on Certain Well-Known Pairwise Independent Random Walks} \label{Appendix}

There exist various well-known pairwise independent hash functions from $[n]$ to $\{\pm 1\}.$ Here, we consider a function generated by a Hadamard matrix and a function generated by a random linear function from $\BZ/p \BZ \to \BZ/p \BZ,$ where $p$ is a prime.

\subsection{Hadamard Matrices} \label{Hadamard}

For $n = 2^k,$ the $n \times n$ Hadamard matrix is a matrix $H$ with each entry equal to $1$ or $-1.$ Suppose we write each integer $0 \le i < n$ as a vector in $\{0, 1\}^k$ corresponding to the binary expansion of $i$. Then, $H_{ij},$ or the element in the $i$th row and $j$th column of $H$ (where we $0$-index the rows and columns) satisfies $H_{ij} = (-1)^{\langle i, j\rangle},$ where $\langle i, j\rangle$ equals the inner product of $i$ and $j$ when written as vectors in $\{0, 1\}^k.$ Now, if we define $h_i: [n-1] \to \{\pm 1\}$ as $h_i(j) := H_{ij},$ then it is well-known and straightforward to verify that drawing $h$ randomly from $\{h_0, h_1, \dots, h_{n-1}\}$ gives us that $\BE_h[h(j)] = 0$ for all $1 \le j \le n-1$ and $\BE_h[h(j) h(j')] = 0$ for all $1 \le j, j' \le n-1,$ $j \neq j'.$ (Note that $h_i(0) = 1$ always, so we make sure that $j \ge 1$ rather than $j \ge 0.$)

We claim that the induced pairwise independent random walk does not satisfy $\BE_h[\sup_i (h(1) + \cdots + h(j))] = \Omega(\sqrt{n} \log n).$ In fact, we have the following:

\begin{proposition}
    Suppose $h$ is drawn uniformly from $h_i$ where $0 \le i \le n-1.$ Then,
\[\BE_{h} \left[\sup\limits_{j} |h(1) + \cdots + h(j)|\right] = \Theta(\log n).\]
\end{proposition}

\begin{proof}
    Since $h(0) = 1$ always, we have that $h(0) + \cdots + h(j) = 1 + h(1) + \cdots + h(j)$ for all $j$. Therefore, it suffices to show that 
\[\BE_{h} \left[\sup\limits_{j} |h(0) + \cdots + h(j)|\right] = \Theta(\log n).\]

    Assume that $h = h_i$ where $i \neq 0,$ and let $0 \le v_2(i) \le k-1$ represent the largest integer $\ell$ such that $2^{\ell}|i.$ Then, note that for any integer $0 \le c < 2^{k-v_2(i)-1},$ we have that 
\[h_i(c \cdot 2^{v_2(i)+1}) = \cdots = h_i(c \cdot 2^{v_2(i)+1} + 2^{v_2(i)}-1) \neq h_i(c \cdot 2^{v_2(i)+1} + 2^{v_2(i)}) = \cdots = h_i(c \cdot 2^{v_2(i)+1} + 2^{v_2(i)+1}-1).\]
    This follows from the fact that $i$, when written in binary, has its last $v_2(i)$ coordinates equal to $0$ but its $(v_2(i)+1)$th to last coordinate equal to $1$. As a result, we have that if $2^{v_2(i)+1}|j,$ then $h(0) + \cdots + h(j) = 0.$ Therefore, as $h(j) \in \{\pm 1\}$ for all $j$, we have that $|h(0) + \cdots + h(j)|$ achieves its maximum when $v_2(j) = v_2(i),$ and equals $2^{v_2(i)}.$ Finally, if $h = h_0,$ then $|h(0) + \cdots + h(j)| \le |h(0) + \cdots + h(n-1)| = n.$
    
    Therefore, with probability $\frac{1}{n},$ we have $\sup_j |h(1) + \cdots + h(j)| = n.$ Also, for each $0 \le \ell \le k-1,$ we have $v_2(i) = \ell$ with probability $2^{-\ell-1},$ in which case $\sup_j |h(1) + \cdots + h(j)| = 2^{\ell}.$ Therefore, we have
\[\BE_h\left[\sup\limits_{j} |h(0) + \cdots + h(j)|\right] = n \cdot \frac{1}{n} + \sum\limits_{\ell = 0}^{k-1} \left(2^{-\ell-1} \cdot 2^{\ell}\right) = 1 + \frac{k}{2} = \Theta(\log n). \qedhere\]
\end{proof}

\subsection{Linear Functions on Finite Fields} \label{Kloosterman}

It is well-known that for any finite field $\BF_q$ the family $f_{a, b}(x): x \mapsto ax+b$ where $a, b$ are randomly chosen from $\BF_q$ is pairwise independent, with $\BP_{a, b}(f_{a, b}(x_1) = y_1, f_{a, b}(x_2) = y_2) = \frac{1}{q^2}$ for all $x_1 \neq x_2 \in \BF_q$ and $y_1, y_2 \in \BF_q.$ As a result, we can create the following function from $[q] \to \{0, \pm 1\}.$ First, create some ordering of the elements of $\BF_q,$ giving us a bijection $g: [q] \to \BF_q.$ Next, if $q$ is a power of $2$, choose some subset $S \subset \BF_q$ of size $q/2$ that will map to $1$ (and the rest will map to $-1$). If $q$ is odd, we will say that $0 \in \BF_q$ maps to $0$ and we choose some subset $S \subset \BF_q \backslash \{0\}$ of size $\frac{q-1}{2}$ to map to $1$. Thus, our final map $f: [q] \to \{0, \pm 1\}$, given $a, b$, will map $x \in [q]$ to $1$ if $a \cdot g(x) + b \in S,$ to $0$ if $q$ is odd and $a \cdot g(x)+b = 0,$ and to $-1$ otherwise. While we really want a random walk with steps in $\{\pm 1\},$ we note that $f_{a, b}$ maps at most one element $x \in \BF_q$ to $0$ unless $a = b = 0,$ this will not affect the expectation of the supremum of the random walk more than $O(1).$

We note that providing general bounds on these random walks is quite difficult, so we will look at the special case where $q$ is an odd prime, the bijection $g$ is reduction modulo $q$, and $S = \left\{1, 2, \dots, \frac{q-1}{2}\right\}.$ In this case, we let $f_{a, b}: [q] \to \{0, \pm 1\}$ be the function such that
\[f(x) = \begin{cases} 0 & ax+b \equiv 0 \Mod q \\ 1 & ax + b \equiv s \Mod q, s \in S \\ -1 & \text{else} \end{cases} \]

In this case, we have the following result, which demonstrates that $f$ will not give a random walk with expected supremum distance $\omega(\sqrt{n})$.

\begin{proposition}
    We have that
\[\BE_{a, b} \left[\sup\limits_{k} \left|f_{a, b}(1) + \cdots + f_{a, b} (k)\right|\right] = O(\log^2 q).\]
\end{proposition}

\begin{remark}
    The solution below follows a very similar approach to approximating the number of solutions to $ab \equiv c \Mod q$ where $c$ is fixed $a, b$ are known to be in some fixed intervals contained in $\BZ/q\BZ,$ using bounds on the Kloosterman sums \cite{Merel96}.
\end{remark}

\begin{proof}
    For simplicity of notation, we will use notation commonly seen in analytic number theory. Namely, for $x \in \BR,$ we will write $e(x) := e^{2 \pi i x}$ and $\|x\| := \min(x - \lfloor x \rfloor, \lceil x \rceil - x),$ or the closest distance between $x$ and some integer.

    Consider the function $h: \BZ/q\BZ \to \BR$ where $h(0) = 0,$ $h(s) = 1$ if $s \in S,$ and $h(s) = -1$ otherwise. We consider the Discrete Fourier Transform of $h$, 
\[\widehat{h}(s) = \sum\limits_{i = 0}^{q-1} h(i) e\left(-\frac{i \cdot s}{q}\right).\]
    It is clear that $\widehat{h}(0) = 0$ and that for $s \neq 0,$
\begin{align*}
|\widehat{h}(s)| &\le \left|\sum\limits_{i = 1}^{(q-1)/2} e\left(-\frac{i \cdot s}{q}\right)\right| + \left|\sum\limits_{i = (q+1)/2}^{q-1} e\left(\frac{i \cdot s}{q}\right)\right| \\
&= 2 \left|\frac{e\left(-\frac{s \cdot (q-1)}{2q}\right)-1}{e\left(-\frac{s}{q}\right)-1}\right| \\
&\le 4 \cdot \frac{1}{\left|e\left(-\frac{s}{q}\right)-1\right|} \\
&= O\left(\frac{1}{\|s/q\|}\right).
\end{align*}

Now, using the discrete Fourier inversion formula, if $a \neq 0,$ we can write
\begin{align*}
    \left|f_{a, b}(1) + \cdots + f_{a, b}(k)\right| &= \left|\sum\limits_{j = 1}^{k} h(j \cdot a + b)\right| \\
    &= \frac{1}{q} \left|\sum\limits_{i = 0}^{q-1} \sum\limits_{j = 1}^{k} \widehat{h}(i) e\left(i \cdot \frac{j \cdot a + b}{q}\right)\right| \\
    &\lesssim \frac{1}{q} \sum\limits_{i = 1}^{q-1} \frac{1}{\|i/q\|} \cdot \left|\sum\limits_{j = 1}^{k} e\left(\frac{(i \cdot a) \cdot j + (i \cdot b)}{q}\right)\right| \\
    &= \frac{1}{q} \sum\limits_{i = 1}^{q-1} \frac{1}{\|i/q\|} \cdot \left|\frac{e\left(\frac{(i \cdot a) \cdot k}{q}\right) - 1}{e\left(\frac{i \cdot a}{q}\right) - 1}\right| \\
    &\lesssim \frac{1}{q} \sum\limits_{i = 1}^{q-1} \frac{1}{\|i/q\|} \cdot \frac{1}{\|i \cdot a/q\|}.
\end{align*}
Note that the final expression is independent of $k$, so we can take a supremum over all $k$. Also, the final expression is independent of $b$. Summing over all values of $b \in \BZ/q \BZ$ and all $a \in (\BZ/q \BZ) \backslash \{0\}$ gives us
\begin{align*}
\sum\limits_{a \neq 0, b} \sup_k \left|f_{a, b}(1) + \cdots + f_{a, b}(k)\right| &\lesssim \sum\limits_{a = 1}^{q-1} \sum\limits_{i = 1}^{q-1} \frac{1}{\|i/q\|} \cdot \frac{1}{\|i \cdot a/q\|} \\
&= \sum\limits_{c = 1}^{q-1} \sum\limits_{i = 1}^{q-1} \frac{1}{\|i/q\|} \cdot \frac{1}{\|c/q\|} \\
&= \left(2 \sum\limits_{i = 1}^{(q-1)/2} \frac{1}{i/q}\right)^2 \\
&\lesssim q^2 \log^2 q.
\end{align*}

    Finally, for $a = 0,$ we have $q$ possibilities of $b$ and we clearly have $\sup |f_{a, b}(1) + \cdots + f_{a, b}(k)| \le q$ as the image of $f$ is contained in $\{0, \pm 1\}.$ Therefore,
\[\sum\limits_{a, b} \left(\sup\limits_{k} |f_{a, b}(1) + \cdots + f_{a, b}(k)|\right) \lesssim q^2 + q^2 \log^2 q \lesssim q^2 \log^2 q.\]
    Thus, taking the expectation over $a, b,$ gives us 
\[\BE_{a, b} \left(\sup\limits_{k} |f_{a, b}(1) + \cdots + f_{a, b}(k)|\right) = O(\log^2 q),\]
    as desired.
\end{proof}

\section{Final Case of Theorem \ref{Main}} \label{FinishFullTheorem}

Here, we prove the final case of Theorem \ref{Main}.

\begin{proposition}
    Let $q \ge p \ge 1,$ $k \ge 2$. Suppose that either $2 \le k \le 3$ and $p \ge 3$, or $k \ge 4$ and either $p > 2 \lfloor \frac{k}{2} \rfloor$ or $p = q = 1$. Then, if $n$ is sufficiently large, there exist unbiased $k$-wise independent random variables $X_1, \dots, X_n$ with $\BE[|X_i|^q] \le 1$ for all $i$ such that
\[\BE\left[|X_1 + \dots + X_n|^p\right] = \Omega(n)^{(p+1)/2}.\]
\end{proposition}

\begin{proof}
    Note that $2 \le k \le 3$ and $p \ge 3$ implies $p > 2 \lfloor \frac{k}{2}\rfloor.$ Therefore, we'll first consider the case of $k \ge 2$ and $p \ge 2 \lfloor \frac{k}{2} \rfloor + 1.$ The variables we will deal with will be $0$ or $\pm 1$ random variables, so $\BE[|X_i|^q] \le 1$ for all $q \ge 1.$
    
    Suppose $n$ is an odd prime satisfying $n \ge 2k$. Let $f: \BF_n \to \BF_n$ be a randomly chosen polynomial of degree at most $k-1,$ i.e. $f(x) = \sum_{i = 0}^{k-1} a_i x^i,$ where $a_0, \dots, a_{k-1}$ are randomly chosen from $\BF_n = \BZ/n\BZ.$ Now, for $1 \le i \le n,$ let $X_i$ be the random variable which equals $1$ if $f(i)$ is a nonzero quadratic residue modulo $n$, $-1$ if $f(i)$ is a quadratic nonresidue modulo $n$, and $0$ of $f(i) = 0.$ It is well known that $f(1), \dots, f(n)$ are $k$-wise independent, so we therefore have that $X_1, \dots, X_n$ are $k$-wise independent. However, there are $n^{\lceil k/2 \rceil}$ polynomials of degree at most $\frac{k-1}{2},$ which means that the number of polynomials $f$ of degree at most $k-1$ which are squares over $\BF_p[x]$ is $\Theta(n^{\lceil k/2 \rceil}),$ since no polynomial can have more than $2$ square roots. In this case, $X_1 + \dots + X_n \ge n-k \ge \frac{n}{2}$ since $n \ge 2k,$ since $f(i)$ is always a quadratic residue, and can be $0$ for at most $k$ values of $i$. Therefore, we have that
\[\BE\left[|X_1 + \dots + X_n|^p\right] \ge \frac{\Theta(n^{\lceil k/2 \rceil})}{n^k} \cdot \Omega(n)^p = \Omega(n)^{p-\lfloor k/2\rfloor}.\]
    However, since $p \ge 2 \lfloor \frac{k}{2} \rfloor + 1,$ we have that $p-\lfloor \frac{k}{2} \rfloor \ge \frac{p+1}{2}$. 
    
    Now, if $n$ is not prime but $n \ge 4k.$ we can choose $n'$ prime so that $2k \le \frac{n}{2} \le n' \le n$ by Bertrand's postulate, and choose random variables $X_1, \dots, X_{n'}$ as in the previous paragraph, and have $X_{n'+1}, \dots, X_n = 0$.
    
    For the case of $p = q = 1$ and $n \ge 1,$ we let $X_1, \dots, X_n$ be i.i.d. drawn from $X$, where $X = 0$ with probability $1 - \frac{1}{n},$ $X = n$ with probability $\frac{1}{2n},$ and $X = -n$ with probability $\frac{1}{2n}.$ It is clear that $\BE[|X|] = 1.$ However, if we have samples $X_1, \dots, X_n,$ exactly one of these samples will be nonzero with probability $n \cdot (1 - \frac{1}{n})^{n-1} \cdot \frac{1}{n},$ which converges to $\frac{1}{e}$ as $n \to \infty.$ In this case, $|X_1 + \dots + X_n| = n.$ Therefore, we have that $\BE[|X_1 + \dots + X_n|] \ge (\frac{1}{e}-o(1)) \cdot n = \Omega(n)$, which concludes the proof.
\end{proof}

\section{Proof of Theorem \ref{Bound3} without the $4$-wise independence criterion} \label{Sheffield}

Here, we prove a weaker version of Theorem \ref{Bound3} where $X_1, \dots, X_n$ are assumed to be totally independent as a direct corollary of the Skorokhod Embedding Theorem.

First, we state the Skorokhod Embedding Theorem (see, for instance, \cite[Theorem 37.7]{Skorokhod}).

\begin{theorem} \label{Sko}
    Let $X_1, \dots, X_n$ be independent unbiased random variables with finite variances, and let $W$ be a Brownian motion process in $1$ dimension. For $1 \le i \le n,$ let $S_i = X_1 + \dots + X_i$ Then, there exists a series of stopping times $tau_0 := 0$ and $0 \le \tau_1 \le \tau_2 \le \dots \le \tau_n$ with respect to the filtration of $W$, such that for all $1 \le i \le n,$ the joint distribution $(W_{\tau_1}, \dots, W_{\tau_n})$ is the same as the distribution $(S_1, \dots, S_n)$, and $\BE[\tau_i-\tau_{i-1}] = \BE[X_i^2].$
\end{theorem}

We will also need the Reflection Principle of $1$-dimensional Brownian motion (see, for instance, \cite[Section 4.2]{ReflectionPrinciple}).

\begin{theorem} \label{Reflection}
    If $W$ is a Brownian motion process in $1$ dimension, then for any time $\tau,$ the distribution of $\sup_{0 \le t \le \tau} W(t)$ has the same distribution as $|W(\tau)|.$
\end{theorem}

We now can prove the weaker version of Theorem \ref{Bound3}.

\begin{proof}
    Let all notation be as in Theorem \ref{Sko}. Then, $\max |S_i|$ has the same distribution as $\max |W_{\tau_i}|.$ Thus, $\BE[\max |S_i|^2] = \BE[\max |W_{\tau_i}|^2] \le \BE[\sup_{0 \le t \le \tau_n} |W_t|^2].$ Now, since $\sup |W_t| \le \sup W_t + \sup (-W_t),$ and since $\sup W_t$ and $\sup (-W_t)$ have the same distribution, we have that
\[\BE\left[\sup_{0 \le t \le \tau_n} |W_t|^2\right] \le \BE\left[\left(\sup_{0 \le t \le \tau_n} W_t + \sup_{0 \le t \le \tau_n} (-W_t)\right)^2\right] \le 4 \BE\left[\left(\sup_{0 \le t \le \tau_n} W_t\right)^2\right].\]
    However, by Theorem \ref{Reflection}, $\sup_{0 \le t \le T} W_t$ has the distribution $|Z|$ where $Z \sim \mathcal{N}(0, T).$ Therefore, conditioning on $\tau_1, \dots, \tau_n,$ we have that 
\[\BE\left[\left(\sup_{0 \le t \le \tau_n} W_t\right)^2\biggr\vert \tau_1, \dots, \tau_n\right] = \BE\left[|\mathcal{N}(0, \tau_n)|^2\right] = \tau_n.\]
    However, by Theorem \ref{Sko}, $\BE[\tau_n] = \BE[X_1^2 + \dots + X_n^2].$ Therefore, we have that
\[\BE\left[\max_{1 \le i \le n} |S_i|^2\right] \le 4 \cdot \sum_{i = 1}^{n} \BE[X_i^2]. \qedhere\]
\end{proof}

\end{document}